\documentclass[11pt,reqno]{amsart}
\usepackage[utf8]{inputenc}
\usepackage[T1]{fontenc}
\usepackage[english]{babel}
\usepackage[usenames,dvipsnames]{xcolor}

\usepackage{hyperref}
\definecolor{darkred}{rgb}{0.7,0,0}
\definecolor{darkblue}{rgb}{0,0,0.7}
\hypersetup{colorlinks=true,urlcolor=darkred,linkcolor=darkred,citecolor=darkblue}
\usepackage{dsfont}
\usepackage{amsfonts}
\usepackage{amssymb}
\usepackage{amsmath}
\usepackage{amsthm}
\usepackage[left=3cm, right=3cm, top=2.2cm,bottom=2.2cm]{geometry}
\usepackage{mathtools}
\usepackage{etoolbox}

\usepackage{tikz}
\usepackage{tikz-cd}
\usetikzlibrary{arrows}
\tikzset{nodes={inner sep=0.2em}}
\tikzset{commutative diagrams/.cd,arrow style=tikz}

\DeclareRobustCommand\myiso{\xrightarrow{\!\smash{\raisebox{-0.25ex}{\ensuremath{\scriptstyle\sim}}\!}}}
\DeclareRobustCommand\myosi{\xleftarrow{\!\smash{\raisebox{-0.25ex}{\ensuremath{\scriptstyle\sim}}\!}}}

\patchcmd{\section}{\scshape}{\bfseries}{}{}
\makeatletter
\renewcommand{\@secnumfont}{\bfseries}
\makeatother

\usepackage{enumitem}
\newenvironment{enu}[1]{\begin{enumerate}[leftmargin=*,labelsep=0.5em,itemsep=0.1em]#1}{\end{enumerate}}

\usepackage{cleveref}

\crefname{thm}{Theorem}{Theorems}
\crefname{prop}{Proposition}{Propositions}
\crefname{thmintro}{Theorem}{Theorems}
\crefname{propintro}{Proposition}{Propositions}
\crefname{lemma}{Lemma}{Lemmas}
\crefname{rem}{Remark}{Remarks}
\crefname{cor}{Corollary}{Corollaries}
\crefname{defi}{Definition}{Definitions}
\crefname{ex}{Example}{Examples}
\crefname{section}{Section}{Sections}

\newtheoremstyle{standard}{2ex}{2ex}{\itshape}{}{\bfseries}{.}{.5em}{}
\theoremstyle{standard}
\newtheorem{lemma}{Lemma}[section]
\newtheorem{prop}[lemma]{Proposition}
\newtheorem{cor}[lemma]{Corollary} 
\newtheorem{thm}[lemma]{Theorem}
\newtheorem{thmintro}{Theorem}[section]

\newtheorem{propintro}[thmintro]{Proposition}

\newtheoremstyle{definition}{2ex}{2ex}{}{}{\bfseries}{.}{.5em}{}    
\theoremstyle{definition}
\newtheorem{defi}[lemma]{Definition}
\newtheorem{rem}[lemma]{Remark} 
\newtheorem{ex}[lemma]{Example}

\newtheoremstyle{intermediate}{2ex}{2ex}{}{}{\itshape}{.}{.5em}{}    
\theoremstyle{intermediate}
\newtheorem*{recol}{Recollection}
\newtheorem*{conv}{Convention}
 
\DeclareMathOperator{\Q}{\mathsf{Qcoh}}
\DeclareMathOperator{\Mod}{\mathsf{Mod}}

\DeclareMathOperator{\Cat}{\mathsf{Cat}}
\DeclareMathOperator{\CAlg}{\mathsf{CAlg}}

\DeclareMathOperator{\Set}{\mathsf{Set}}
\DeclareMathOperator{\Sch}{\mathsf{Sch}}
\DeclareMathOperator{\Ch}{\mathsf{Ch}}

\DeclareMathOperator{\Spec}{Spec}
\DeclareMathOperator{\End}{End}
\DeclareMathOperator{\Hom}{Hom}
\DeclareMathOperator{\HOM}{\underline{\Hom}}
\DeclareMathOperator{\Ind}{Ind}
\DeclareMathOperator{\id}{id}
\DeclareMathOperator{\im}{im}
\DeclareMathOperator{\op}{op}

\DeclareMathOperator{\fp}{fp}
\DeclareMathOperator{\Sym}{Sym}

\renewcommand{\c}{\mathrm{c}}
\DeclareMathOperator{\fc}{fc}

\newcommand\mydot{\parbox[b]{0mm}{.\hspace*{-2mm}}}
\DeclareMathOperator{\coboxtimes}{\mathbin{\widehat{\boxtimes}}}

\newcommand{\A}{\mathcal{A}}
\newcommand{\B}{\mathcal{B}}
\newcommand{\C}{\mathcal{C}}
\newcommand{\D}{\mathcal{D}}

\newcommand{\K}{\mathcal{K}}
\renewcommand{\L}{\mathcal{L}}

\renewcommand{\O}{\mathcal{O}}
\newcommand{\IN}{\mathds{N}}
\newcommand{\IZ}{\mathds{Z}}
\newcommand{\IA}{\mathds{A}}
\newcommand{\IK}{\mathds{K}}
\newcommand{\IP}{\mathds{P}}

\title[Localizations of tensor categories and fiber products of schemes]{Localizations of tensor categories\\and fiber products of schemes}
\author{Martin Brandenburg}
\thanks{\emph{E-mail address:} \texttt{brandenburg@uni-muenster.de}}

\begin{document}
%
\begin{abstract}
\noindent 
We prove that the tensor category of quasi-coherent modules $\Q(X \times_S Y)$ on a fiber product of quasi-compact quasi-separated schemes is the bicategorical pushout of $\Q(X)$ and $\Q(Y)$ over $\Q(S)$ in the \mbox{$2$-category} of cocomplete linear tensor categories. In particular, $\Q(X \times Y)$ is the bicategorical coproduct of $\Q(X)$ and $\Q(Y)$. For this we introduce idals, which can be seen as non-embedded ideals, and use them to study localizations of cocomplete tensor categories in general. 
\end{abstract}
\maketitle
%
\section{Introduction}
%
It is a common theme in algebraic geometry that a scheme or stack $X$ is approached via its (tensor) category of quasi-coherent $\O_X$-modules $\Q(X)$ or variants thereof such as its derived category $D(X)$ as a triangulated category or as an $\infty$-category. For this it is necessary to prove reconstruction results \cite{Gab62,Ros95,BO01,Bal02,Lur05,GC15,Ant16,Bra18} and to find (tensor) categorical properties and constructions which correspond to geometric ones \cite{Ros95,Bal10,Sch18,Bra14}.

For example, if $X$ is a quasi-compact quasi-separated scheme and $Y$ is an arbitrary scheme, then any cocontinuous tensor functor $F : \Q(X) \to \Q(Y)$ is induced by a unique morphism $f : Y \to X$ via pullback \cite{BC14}; similar results hold for well-behaved algebraic spaces and algebraic stacks and are usually referred to as Tannaka reconstruction theorems \cite{Lur05,Lur11,Sch12,FI13,Ton14,Bha16,BHL17,HR19}. Moreover, geometric properties of~$f$, such as being affine or projective, can be formulated in terms of $F$ \cite{Bra14,Sch18}.

It should also be expected, and this is what one actually needs in order to obtain the aforementioned results, that a universal property of a scheme $X$, say over some commutative ring $\IK$, corresponds to a bicategorical universal property of $\Q(X)$ within the $2$-category of all cocomplete $\IK$-linear tensor categories, not just those of the form $\Q(Y)$; we refer to \cite{Ben67,KR74} for basic bicategorical concepts. This has been achieved for plenty of examples in the author's PhD thesis \cite{Bra14}. Perhaps the most elementary example is the observation that $\Spec(\IK)$ is the final $\IK$-scheme, and $\Q(\Spec(\IK)) \simeq \Mod(\IK)$ is, in fact, the initial cocomplete $\IK$-linear tensor category in the bicategorical sense. A more involved example is the observation that if $X$ is a projective $\IK$-scheme, then $\Q(X)$ satisfies a similar universal property as $X$, namely it classifies invertible objects with suitable global generators satisfying some prescribed relations \cite[Section 5.10]{Bra14}.

In this paper we are interested in the question whether $X \mapsto \Q(X)$ preserves fiber products. That is, if $X,Y$ are $S$-schemes, we have a square of cocontinuous tensor functors
\[\begin{tikzcd}[row sep=6ex, column sep=4ex]
\Q(S) \ar{d} \ar{r} & \Q(X)  \ar{d} \\ \Q(Y) \ar{r} & \Q(X \times_S Y)
\end{tikzcd}\]
which commutes up to a specified isomorphism, and our goal is to prove that this is actually a bicategorical pushout square if $X,Y,S$ are quasi-compact quasi-separated. This is a basic result which is necessary for the translation of those geometric notions to tensor category theory which involve fiber products or base changes, such as separateness, properness and algebraic correspondences.

The pushout property has already been proven by the author if $X$ is quasi-projective over~$S$ \cite{Bra14}, using the  universal property of $\Q(X)$ mentioned before. It has also been proven by Sch\"appi if $X,Y,S$ are quasi-compact semi-separated algebraic stacks with the resolution property \cite{Sch18} (a special case appeared earlier in \cite{Sch14}), using a generalization of Tannaka duality. Moreover, Ben-Zvi, Francis and Nadler have proven such a result for perfect stacks in the $\infty$-categorical setting \cite{BFN10}, which has subsequently been generalized by Lurie \cite{Lur18}.

Our main result states the following.

\begin{thmintro}\label{intro-1}
Let $\IK$ be a commutative ring and let $X$ and $Y$ be two quasi-compact quasi-separated $\IK$-schemes over some quasi-compact quasi-separated $\IK$-scheme $S$. Then the square above exhibits $\Q(X \times_S Y)$ as a bicategorical pushout of $\Q(X)$ and $\Q(Y)$ over $\Q(S)$ in the $2$-category of cocomplete $\IK$-linear tensor categories and cocontinuous $\IK$-linear tensor functors. In particular, $\Q(X \times_\IK Y)$ is a bicategorical coproduct of $\Q(X)$ and $\Q(Y)$.
\end{thmintro}

The last statement can also be written as
\[\Q(X \times_\IK Y) \simeq \Q(X) \coboxtimes_{\IK} \Q(Y),\]
where $\smash{\coboxtimes_{\IK}}$ denotes the tensor product of locally presentable $\IK$-linear categories (cf.\ \cite[Corollary 2.2.5]{CJF13}). It is also equivalent to
\[\Q_{\fp}(X \times_\IK Y) \simeq \Q_{\fp}(X) \boxtimes_{\IK} \Q_{\fp}(Y),\]
where $\boxtimes_{\IK}$ denotes Kelly's tensor product of essentially small finitely cocomplete $\IK$-linear categories (see \cite[Section 6.5]{Kel05} or \cite[Theorem 7]{LF13}). Roughy, it means that quasi-coherent $\O_{X \times_{\IK} Y}$-modules are freely generated under colimits and tensor products by the pullbacks of quasi-coherent $\O_X$-modules and quasi-coherent $\O_Y$-modules. We also prove a generalization of this theorem where $\Q(Y)$ is replaced by a suitable tensor category $\C$ and $\Q(X \times_{\IK} Y)$ by a certain tensor category $\Q_{\C}(X)$ of \emph{quasi-coherent $\O_X$-modules internal to $\C$} (\cref{genprodthm}). 
  
Since a scheme is built up out of affine pieces, it is tempting to reduce this theorem to the affine case. For this we will have to find well-behaved tensor categorical analogs of open subschemes. The first idea might be to use ideals and localize at them. In fact, if $I \subseteq \O_X$ is a quasi-coherent ideal, then $X_I \coloneqq  \{x \in X : I_x = \O_{X,x}\}$ is an open subscheme of $X$, and every open subscheme of $X$ has this form. Thus, when $\C$ is a suitable tensor category with unit object $\O_\C$ and $I \subseteq \O_\C$ is an ideal, i.e.\ a subobject, we might define a localization $\C_I$ of~$\C$ at $I$. But if $F : \C \to \D$ is a cocontinuous tensor functor, then $F$ does not have to preserve ideals. This makes ideals unsuitable for our theory.

We solve this problem by working with \emph{idals} instead, which may be seen as non-\underline{e}mbedded id\underline{e}als (not to be confused with ideles from algebraic number theory). An idal in $\C$ is just a morphism $e : I \to \O_\C$ satisfying the equation $e \otimes I = I \otimes e$ in $\Hom(I \otimes I,I)$. Thus, idals are preserved by any tensor functor whatsoever. In the case $\C=\Q(X)$ every idal $I \to \O_X$ induces an open subscheme $X_I \coloneqq \{x \in X : I_x \to \O_{X,x} \text{ is an isomorphism}\}$ of $X$.

\begin{thmintro}\label{intro-2} 
Let $\C$ be a locally presentable tensor category and let $I \to \O_\C$ be an idal in $\C$. Then there is a localization $\C_I$ of $\C$ at $I \to \O_\C$ in the $2$-category of cocomplete tensor categories. The underlying category of $\C_I$ is the full reflective subcategory of $\C$ containing those $M \in \C$ for which $M \myiso \HOM(\O_\C,M) \to \HOM(I,M)$ is an isomorphism. Moreover, if $\C=\Q(X)$ for some quasi-compact quasi-separated scheme $X$ and $I$ is of finite presentation, then we have
\[\Q(X)_I \simeq \Q(X_I).\]
\end{thmintro}
 
The latter statement about schemes uses a new variant of Deligne's formula \cite[Appendix, Proposition 4]{Har66} which also works for non-noetherian schemes (basically since we work with tensor powers instead of ideal powers).

\begin{propintro}\label{intro-3}
Let $X$ be a quasi-compact quasi-separated scheme and let $J \to \O_X$ be a quasi-coherent idal of finite presentation. If $j : X_J \to X$ denotes the open immersion, then for any $M \in \Q(X)$ we have a natural isomorphism
\[{\varinjlim}_n \HOM(J^{\otimes n},M) \myiso j_* j^* M.\]
In particular, we have a natural isomorphism
\[{\varinjlim}_n \Hom(J^{\otimes n},M) \myiso \Gamma(X_J,M).\]
\end{propintro}
 
Incidentally this formula can also be used to deduce that the finitely presentable idals generate $\Q(X)$ as a cocomplete tensor category (\cref{idalgenerate}).
 
Let us say that two idals $I \to \O_\C \leftarrow J$ in a tensor category $\C$ form a \emph{cover} if the square
\[\begin{tikzcd}[row sep=6ex, column sep=6ex] I \otimes J \ar{r}{e \otimes J}  \ar{d}[swap]{I \otimes f} & J \ar{d}{f} \\ I \ar{r}[swap]{e} & \O_\C \end{tikzcd}\]
is a pushout. For example, if $\C=\Q(X)$, this corresponds to an open cover $X = X_I \cup X_J$ in the usual sense, and we know that quasi-coherent modules may be glued, i.e.\ that the restriction functors induce an equivalence
\[\Q(X) \myiso \Q(X_I) \times_{\Q(X_I \cap X_J)} \Q(X_J).\]
The latter is a bicategorical pullback. Such a gluing result is also true for well-behaved tensor categories.
 
\begin{thmintro}\label{intro-4}
Let $\C$ be a locally finitely presentable tensor category. Let $I \to \O_\C \leftarrow J$ be two finitely presentable idals which form a cover. Then there is a natural equivalence
\[\C \myiso \C_I \times_{\C_{I \otimes J}} \C_J.\]
\end{thmintro}

\cref{intro-2,intro-4} are the main ingredients of the proof of \cref{intro-1}. They also allow us to find, for every quasi-compact quasi-separated $\IK$-scheme $X$, a bicategorical universal property of $\Q(X)$ within the $2$-category of locally finitely presentable $\IK$-linear tensor categories with cocontinuous $\IK$-linear tensor functors preserving finitely presentable objects (see \cref{sec:localtens}). We consider for example the affine line (resp.\ plane) with a double origin and the projective line in \cref{doubleorigin,projline,doubleorigin2}.

The paper is organized as follows. In \cref{sec:idals} we introduce idals and make some basic observations about them. In \cref{sec:open} we study idal covers and open subschemes associated to idals in $\Q(X)$. In \cref{sec:deligne} we prove \cref{intro-3}, the variant of Deligne's formula. In \cref{sec:tensorloc} we study the localization $\C_I$ of a cocomplete tensor category at an idal $I$ and prove \cref{intro-2,intro-4}. In \cref{sec:fiber} we combine all this to prove \cref{intro-1}. In \cref{sec:localtens} we give a local description of cocontinuous tensor functors with respect to an idal cover and apply this to describe cocontinuous tensor functors on $\Q(X)$, which yields another proof of \cref{intro-1}.\\

\noindent\textbf{Ackowledgements.} I would like to thank David Rydh for making many detailed comments which improved the paper.

\tableofcontents
\section{Idals}
\label{sec:idals}

\begin{conv}
For us, a \emph{tensor category} is a symmetric monoidal category, and a \emph{tensor functor} is a strong symmetric monoidal functor \cite[Chapter XI]{ML98}. The unit object of a tensor category $\C$ will be denoted by $\O_\C$. This notation is motivated by algebraic geometry.
\end{conv}

\begin{recol}
An ideal of a commutative ring $\IK$ is just a monomorphism of $\IK$-modules $I \hookrightarrow \IK$. Similarly, if $X$ is a scheme, a quasi-coherent ideal of $\O_X$ is just a monomorphism of quasi-coherent $\O_X$-modules $I \hookrightarrow \O_X$. More generally, if $\C$ is a tensor category, an ideal of $\O_\C$ can be defined as a monomorphism $I \hookrightarrow \O_\C$ in $\C$.
\end{recol}
 
One drawback of ideals is that they do not pull back nicely. If $F : \C \to \D$ is a tensor functor and $I  \to \O_\C$ is an ideal, then its image $F(I) \to F(\O_\C) \myiso \O_\D$ is not an ideal in general, since~$F$ does not have to preserve monomorphisms. For example, if $f : X \to Y$ is a morphism of well-behaved schemes, then the induced pullback functor $f^* : \Q(Y) \to \Q(X)$ preserves ideals only if $f$ is flat. Back to the general case, if $\D$ has image factorizations, we may consider the image $F(I) \cdot \O_\D$ of $F(I) \to \O_\D$, which is an ideal of $\O_\D$. To some extent, this can be used to do commutative algebra with ideals and prime ideals in tensor categories \cite[Section~4.2]{Bra14}. It can even be used to introduce the blow-up of a locally presentable tensor category along an ideal \cite[Section 5.10.2]{Bra14}. However, this workaround is not adequate for our purposes, and in fact the restriction to monomorphisms is not necessary. As a substitute, we will now introduce idals.

\begin{defi}
Let $\C$ be a tensor category with unit object $\O_\C$. An \emph{idal} in $\C$ is a morphism
\[e : I \to \O_\C\]
such that the diagram
\[\begin{tikzcd}[row sep=6ex, column sep=6ex] I \otimes I \ar{r}{e \otimes I} \ar{d}[swap]{I \otimes e} & \O_\C \otimes I \ar{d}{\sim} \\ I \otimes \O_\C \ar{r}[swap]{\sim} & I \end{tikzcd}\]
commutes. Viewing the coherence isomorphisms $\O_\C \otimes I \myiso I \myosi I \otimes \O_\C$ as identities (making use of the coherence theorem \cite[Theorem XI.1.1]{ML98}), we may write the condition as
\[e \otimes I = I \otimes e : I \otimes I \to I.\]
It is sometimes useful to abuse notation and abbreviate an idal $e : I \to \O_\C$ by $I$.
\end{defi}

Thus, in our main example $\C=\Q(X)$ for some scheme $X$, an idal is a homomorphism $e : I \to \O_X$ of quasi-coherent $\O_X$-modules such that $e(x) \cdot y = e(y) \cdot x$ holds for all local sections $x,y$ of $I$.
 
\begin{rem} \label{idal-prop}
Idals enjoy the following properties:
\begin{enu}
\item Notice that $e \circ (e \otimes I ) = e \otimes e = e \circ (I \otimes e)$ holds for every morphism $e : I \to \O_\C$. In particular, if $e : I \to \O_\C$ is a monomorphism, i.e.\ the inclusion of an ideal, then $e$ will be an idal. In general, an idal is like an id\underline{e}al which is not necessarily \underline{e}mbedded.
\item\label{idalsymtr} We have $e \otimes I = (I \otimes e) \circ S_{I,I}$ for the self-symmetry $S_{I,I} : I \otimes I \to I \otimes I$. In particular, $e \otimes I = I \otimes e$ holds when $S_{I,I}$ is the identity, i.e.\ $I$ is symtrivial in the sense of \cite[Section 4.3]{Bra14}. In particular, every morphism $e : \O_\C \to \O_\C$ is an idal. This provides examples of idals which are no ideals.
\item If $F : \C \to \D$ is any tensor functor and $e : I \to \O_\C$ is an idal in $\C$, then its image $F(e) : F(I) \to F(\O_\C) \myiso \O_\D$ is an idal in $\D$.
\item There is an obvious notion of a morphism of idals in $\C$. In fact, idals form a full subcategory of the slice category $\C / \O_\C$. In contrast to ideals, the category of idals is usually not a~preorder.
\item \label{transit} If $e : I \to \O_\C$ is an idal, then for every $n \geq m$ there is a natural morphism $I^{\otimes n} \to I^{\otimes m}$ given by a tensor product of $n-m$ copies of $e$ at any $n-m$ chosen positions and $m$ copies of $\id_I$. This is well-defined because $e$ is an idal.
\item If $e : I \to \O_\C$ is an idal, then $I$ becomes a commutative non-unital algebra object in $\C$ with respect to the multiplication $e \otimes I = I \otimes e : I \otimes I \to I$, and $e$ becomes a morphism of non-unital algebras. This generalizes the fact that an ideal of an algebra can be seen as a~non-unital algebra.
\end{enu} 
\end{rem}

\begin{rem} \label{universalidal}
\cref{idal-prop}(\ref{transit}) tells us how to construct the universal example of a tensor category with an idal. First, we consider the universal example of a tensor category with an object $I$. Its objects are $I^{\otimes n}$ for $n \geq 0$, the only morphisms are $\Hom(I^{\otimes n},I^{\otimes n}) = \Sigma_n$, where $\Sigma_n$ denotes the symmetric group on $n$ letters. We want to have morphisms $I^{\otimes n} \to I^{\otimes m}$ for $n \geq m$ which coequalize all symmetry automorphisms $X^{\otimes n}$. Thus, we redefine $\Hom(I^{\otimes n},I^{\otimes m}) \coloneqq \Sigma_m$ for $n \geq m$ and $\Hom(I^{\otimes n},I^{\otimes m}) \coloneqq \emptyset$ otherwise.
\end{rem}

\begin{rem}
Apart from ideals, idals are also connected to the following concepts:
\begin{enu}
\item The definition of an idal can be formulated in every monoidal category and can also be dualized. Thus, a \emph{coidal} in a monoidal category is a morphism $a : \O_\C \to I$ with $a \otimes I = I \otimes a$. For the monoidal category of endofunctors of a category, this concept is known as a well-pointed endofunctor \cite[Chapter II]{Kel80} and has been studied a lot. Therefore, idals could be also called \emph{co-well-pointed objects}.
\item If $e : I \to \O_\C$ is an open idempotent, i.e.\ a closed idempotent in $\C^{\op}$ in the sense of \cite[Definition 2.8]{BD14}, then $e$ is an idal by \cite[Lemma 2.10]{BD14}. A similar notion of idempotents has been used in tensor triangular geometry \cite[Section 3]{BF11}.
\end{enu}
\end{rem}

\begin{defi}
If $e : I \to \O_\C$ and $f : J \to \O_\C$ are two idals in a tensor category $\C$, then we may define their \emph{idal product} as the composition
\[e \otimes f : I \otimes J \to \O_\C \otimes \O_\C \myiso \O_\C.\]
The following computation shows that this is indeed an idal.
\begin{align*}
e \otimes f \otimes I \otimes J &= (f \otimes I \otimes J) \circ (e \otimes J \otimes I \otimes J) \\
& = (I \otimes f \otimes J)  \circ (S_{J,I} \otimes J) \circ (e \otimes J \otimes I \otimes J) \\
& = (I \otimes J \otimes f) \circ (e \otimes I \otimes J \otimes J) \circ (I \otimes S_{J,I} \otimes J) \\
& = (I \otimes J \otimes f) \circ (I \otimes e \otimes J \otimes J) \circ (I \otimes S_{J,I} \otimes J) \\
& = (I \otimes J \otimes f) \circ (I \otimes J \otimes e \otimes J) \\
& = I \otimes J \otimes e \otimes f
\end{align*}
\end{defi}

\begin{conv}
A \emph{(finitely) cocomplete tensor category} is a tensor category whose underlying category is (finitely) cocomplete in such a way that $\otimes$ preserves (finite) colimits in each variable. Such tensor categories may be seen as categorified rigs \cite{BD98,CJF13}. A tensor category is called $\IK$-linear if its underlying category is $\IK$-linear and $\otimes$ is $\IK$-linear in each variable.
\end{conv}

\begin{lemma} \label{idal-reflect}
If $\C$ is a finitely cocomplete tensor category, then the category of idals in $\C$ is a~reflective subcategory of $\C / \O_\C$.
\end{lemma}

\begin{proof}
Let $f : A \to \O_\C$ be any morphism. We define $\pi : A \to I$ as the coequalizer of the pair of morphisms $f \otimes A, A \otimes f : A \otimes A \rightrightarrows A$. Since they are coequalized by $f$, there is a morphism $e : I \to \O_\C$ with $e \circ \pi = f$, i.e.\ $\pi$ is a morphism $f \to e$ in $\C/ \O_\C$. Then $e$ is an ideal since
\[(e \otimes I) \circ (\pi \otimes \pi) = f \otimes \pi = \pi \circ (f \otimes A) = \pi \circ (A \otimes f) = \pi \otimes f =  (I \otimes e) \circ (\pi \otimes \pi)\]
and $\pi \otimes \pi$ is an epimorphism. The universal property of $\pi$ is easily verified.
\end{proof}

\begin{rem}\label{idal-coproduct}
If $\C$ is a (finitely) cocomplete tensor category, then the category of idals in $\C$ is (finitely) cocomplete by \cref{idal-reflect}, since $\C / \O_\C$ is (finitely) cocomplete. For example, the coproduct of two idals $e : I \to \O_\C$ and $f : J \to \O_\C$ is given by the pushout $P$ of $e \otimes J : I \otimes J \to J$ and $I \otimes f : I \otimes J \to I$ equipped with the morphism $P \to \O_\C$ which extends $f$ and $e$.
\end{rem}

\begin{rem}\label{idal-classif}
Assume that $\C$ is a finitely cocomplete $\IK$-linear tensor category, $J \hookrightarrow \O_\C$ is an ideal and $M \in \C$ is any object such that $J \otimes M \to M$ is zero. Then by \cref{idal-coproduct} the composition $J \oplus M \twoheadrightarrow J \hookrightarrow \O_\C$ is an idal. When $\C=\Mod(R)$ for some Dedekind domain~$R$, every idal $e : I \to R$ is isomorphic to such an idal. In fact, since the ideal $e(I) \subseteq R$ is projective, $\ker(e)$ is a direct summand of $I$, say $I = J \oplus \ker(e)$. Then $e$ induces an isomorphism $J \myiso e(I)$, and the idal property implies $e(I) \cdot \ker(e)=0$.
\end{rem}

\section{Open subschemes and idal covers}
\label{sec:open}

In this section we associate to every idal in $\Q(X)$ an open subscheme of $X$. We also introduce idal covers of tensor categories.

\begin{lemma} \label{isiso}
Let $e : I \to \O_\C$ be an idal in a tensor category $\C$.
\begin{enu}
\item If $e$ is a regular epimorphism, then $e$ is an isomorphism.
\item\label{justepi} If $\C=\Q(X)$ for some scheme $X$ and $e$ is an epimorphism, then $e$ is an isomorphism.
\end{enu}
\end{lemma}

\begin{proof}
If $e$ is a regular epimorphism, then by \cite[Lemma 4.8.9]{Bra14} $e$ is the coequalizer of the pair $e \otimes I, I \otimes e : I \otimes I \rightrightarrows I$. Since these morphisms agree, $e$ must be an isomorphism. The second statement follows from the first since $\Q(X)$ is an abelian category.
\end{proof}

\begin{ex} \label{counter}
\cref{isiso}(\ref{justepi}) does not hold for arbitrary tensor categories. Consider for example the cocomplete tensor category $\C$ of torsion-free abelian groups with the usual tensor product $\otimes_{\IZ}$ and $\O_\C=\IZ$ \cite[Example 5.8.1]{Bra14}. Then $2 : \IZ \to \IZ$ is an epimorphism in $\C$, and it is an idal by \cref{idal-prop}(\ref{idalsymtr}). However, it is no isomorphism.
\end{ex}
 
\begin{defi} \label{XIdef}
Let $X$ be a scheme and let $e : I \to \O_X$ be an idal in $\Q(X)$. We define the subset $X_I \subseteq X$ by
\[X_I \coloneqq \bigl\{x \in X : e_x : I_x \to \O_{X,x} \text{ is an isomorphism}\bigr\}.\]
\end{defi}

\begin{rem}
Since $e_x$ is an isomorphism if and only if $e_x$ is an epimorphism by \cref{isiso} applied to $\Spec(\O_{X,x})$, we have $X_I = X_J$, where $J \coloneqq \im(e) \subseteq \O_X$ is a quasi-coherent ideal and $X_J = \{x \in X : J_x = \O_{X,x}\}$ is the usual open subscheme associated to $J$. In particular, $X_I$ is an open subscheme of $X$.
\end{rem}

\begin{rem}\noindent \label{Xprod}
\begin{enu}
\item\label{Xprodinter} For two idals $I \to \O_X \leftarrow J$ we have $X_{I \otimes J} = X_I \cap X_J$.
\item\label{XIpull} If $f : Y \to X$ is a morphism and $I \to \O_X$ is an idal, then $f^{-1}(X_I) = X_{f^*(I)}$.
\item If $I$ is of finite type and $X$ is quasi-compact, then $X_I$ is quasi-compact.
\end{enu}
\end{rem}

\begin{rem} \label{ueopen}
Let $X$ be a scheme and let $e : I \to \O_X$ be an idal in $\Q(X)$. Then a~morphism $f : Y \to X$ factors through $X_I \subseteq X$ if and only if $f^* : \Q(X) \to \Q(Y)$ maps $e$ to an isomorphism. In fact, by \cref{isiso} $f^*(e)$ is an isomorphism if and only if it is an epimorphism, i.e.\ $f^*(\O_X/J)=0$, where $J$ is the image of $e$. By \cite[Chap. 0, 5.2.4.1]{EGAI} this is equivalent to $f^{-1}(X \setminus X_J)=\emptyset$, which means that $f(Y) \subseteq X_J=X_I$ as sets.
\end{rem}

In the following we will often work with finitely presentable idals $I \to \O_\C$, which shall simply mean that $I$ is finitely presentable.

\begin{lemma} \label{idals-qcqs}
Let $X$ be a quasi-compact quasi-separated scheme. Then, for every quasi-compact open subscheme $U \subseteq X$ there is some finitely presentable idal $I \to \O_X$ such that $U = X_I$.
\end{lemma}

Notice that in general, when $X$ is not noetherian, we will not be able to find a finitely presentable \emph{ideal} with this property. This is yet another advantage of idals.

\begin{proof}
Since $U$ is quasi-compact and every quasi-coherent ideal is the sum of its quasi-coherent subideals of finite type \cite[Corollaire 6.9.9]{EGAI}, we find some quasi-coherent ideal $J \subseteq \O_X$ of finite type such that $U = X_J$. There is some quasi-coherent $\O_X$-module of finite presentation~$P$ which admits an epimorphism $P \twoheadrightarrow J$ \cite[Proposition 6.9.10]{EGAI}. Let $I \to \O_X$ be the reflection of $P \twoheadrightarrow J \hookrightarrow \O_X$ into the category of idals as in \cref{idal-reflect}. Then $I$ is also of finite presentation by construction, and we obtain an epimorphism of idals $I \to J$ so that $X_I = X_J$.
\end{proof}

We now introduce covers of tensor categories. For simplicity, we will restrict to covers with only two elements. Because of the following \cref{gen-qcqs} this will be sufficient for our purposes. This is convenient because we will not have to pay attention to any cocycle conditions. 

\begin{rem} \label{gen-qcqs}
Let $\A$ be a class of schemes with the following two properties: Firstly, $\A$~contains all affine schemes. Secondly, if $X$ is a quasi-compact quasi-separated scheme which is covered by quasi-compact open subschemes $X_1$ and $X_2$ such that $X_1$, $X_2$ and $X_1 \cap X_2$ are contained in $\A$, then $X$ is contained in $\A$. Then, $\A$ contains all quasi-compact quasi-separated schemes. In fact, one first proves that $\A$ contains all quasi-compact open subschemes of affine schemes, using an induction on the number of basic-open subsets from an open cover. The general case then uses an induction on the number of affine schemes from an open cover.
\end{rem}
 
\begin{defi}\label{coverdef}
Let $\C$ be a finitely cocomplete tensor category. An \emph{idal cover} of $\C$ consists of two idals $e : I \to \O_\C$ and $f : J \to \O_\C$ such that the following commutative square is a pushout in $\C$.
\[\begin{tikzcd}[row sep=6ex, column sep=6ex] I \otimes J \ar{r}{e \otimes J}  \ar{d}[swap]{I \otimes f} & J \ar{d}{f} \\ I \ar{r}[swap]{e} & \O_\C \end{tikzcd}\]
By \cref{idal-coproduct} this means that $\id : \O_\C \to \O_\C$ is the coproduct of $e$ and $f$ in the category of idals in $\C$. Clearly, any finitely cocontinuous tensor functor preserves idal covers.
\end{defi}

\begin{lemma}\label{covercrit}
Consider two idals $e : I \to \O_\C$ and $f : J \to \O_\C$ in a finitely cocomplete tensor category $\C$.
\begin{enu}
\item These idals form an idal cover if and only if $(e,f) : I \oplus J \to \O_\C$ is a regular epimorphism.
\item When $\C=\Q(X)$ for some scheme $X$ and $(e,f) : I \oplus J \to \O_X$ is an epimorphism, then these idals form an idal cover, and we have an open covering $X = X_I \cup X_J$.
\end{enu}
\end{lemma}

\begin{proof}
1. The direction $\Longrightarrow$ is trivial. To prove $\Longleftarrow$, assume that $(e,f) : I \oplus J \to \O_\C$ is a~regular epimorphism. By \cite[Lemma 4.8.9]{Bra14} $(e,f)$ is the coequalizer of the pair
\[(e,f) \otimes (I \oplus J),\, (I \oplus J) \otimes (e,f) : (I \oplus J)^{\otimes 2} \rightrightarrows I \oplus J.\]
Under the canonical isomorphism $I^{\otimes 2} \oplus J^{\otimes 2} \oplus (I \otimes J) \oplus (J \otimes I) \myiso (I \oplus J)^{\otimes 2}$ these two morphisms correspond to (where $\iota_I$ and $\iota_J$ denote  the coproduct inclusions)
\[\bigl(\iota_I \circ (e \otimes I),\,\iota_J \circ (f \otimes J),\,\iota_J \circ (e \otimes J),\,\iota_I \circ (f \otimes I)\bigr)\]
and
\[\bigl(\iota_I \circ (I \otimes e),\,\iota_J \circ (J \otimes f),\,\iota_I \circ (I \otimes f),\,\iota_J \circ (J \otimes e)\bigr).\]
Because of $e \otimes I = I \otimes e$ and $f \otimes J = J \otimes f$, it follows that $(e,f)$ is the coequalizer of the pair
\[\iota_J \circ (e \otimes J),\, \iota_I \circ  (I \otimes f) : I \otimes J \to I \oplus J.\]
This exactly means that the square in \cref{coverdef} is a pushout.

2. The first statement follows from (1) since $\Q(X)$ is an abelian category. If $x \in X$, then $(e_x,f_x) : I_x \oplus J_x \to \O_{X,x}$ is surjective. Hence, the images of $e_x$ and $f_x$ cannot be both contained in the unique maximal ideal, and therefore $e_x$ or $f_x$ is surjective. This shows that $X = X_I \cup X_J$. From this one can also directly deduce that $I \oplus_{I \otimes J} J \to \O_X$ is an isomorphism.
\end{proof}
 
\begin{ex}
In \cref{covercrit} it is important to demand that $(e,f)$ is a \emph{regular} epimorphism. In fact, for the idal $I$ in \cref{counter} $I \oplus 0 \to \O_\C$ is an epimorphism, but $(I,0)$ is not an idal cover. For a less pathological counterexample, consider for a field $\IK$ the cocomplete $\IK$-linear tensor category $\C$ of $\IK[X,Y]$-modules $M$ for which $(X;Y) : M \to M^2$ is injective (cf.\ \cite[Example 4.8.12]{Bra14}). The unit object is $\O_\C=\IK[X,Y]$; the tensor product is more complicated. By construction $(X,Y) : \O_\C \oplus \O_\C \to \O_\C$ is an epimorphism. But it is not a regular epimorphism (and hence no idal cover): Otherwise, \cite[Lemma 4.8.9]{Bra14} would imply that it is the cokernel of $(Y;-X) : \O_\C \to \O_\C \oplus \O_\C$, which is not the case.
\end{ex}

\begin{rem} \label{getall}
Let $X$ be a scheme and let $X = U \cup V$ be an open covering. Then the vanishing ideals $I(X \setminus U), I(Y \setminus V) \subseteq \O_X$ form an idal cover. This is because $I(X \setminus U)|_U = \O_U$ and $I(Y \setminus V)|_V = \O_V$, so that $I(X \setminus U) \oplus I(Y \setminus V) \to \O_X$ is an epimorphism. If $X$ is quasi-compact quasi-separated and $U,V$ are quasi-compact, then we may even write $U = X_I$ and $V = X_J$ for idals $I \to \O_X \leftarrow J$ of finite presentation (\cref{idals-qcqs}). These form an idal cover, because $I \oplus J \to \O_X$ is an epimorphism both on $X_I$ and on $X_J$. Besides, $X$ is the pushout of $X_I$ and $X_J$ over $X_I \cap X_J = X_{I \otimes J}$. Therefore, it follows from the proof of \cref{gen-qcqs} that every quasi-compact quasi-separated scheme is built up out of affine schemes using finitely many gluings with respect to finitely presentable idal covers.
\end{rem}

\begin{lemma} \label{covprod}
Let $I,J,J'$ be three idals in a finitely cocomplete tensor category $\C$. Assume that $(I,J)$ and $(I,J')$ are idal covers of $\C$. Then $(I,J \otimes J')$ is an idal cover of $\C$ as well.
\end{lemma}

\begin{proof}
Consider the following commutative diagram:
\[\begin{tikzcd}[row sep=6ex, column sep=6ex]
I \otimes J \otimes J' \ar{r} \ar{d} & I \otimes J' \ar{d} \ar{r} & I \ar{d} \\
J \otimes J' \ar{r} & J' \ar{r} & \O_\C\mydot
\end{tikzcd}\]
The square on the left is a pushout since $(I,J)$ is a cover and tensoring with $J'$ preserves pushouts. The square on the right is a pushout since $(I,J')$ is a cover. It follows that the outer rectangle is a pushout as well. But this means that $(I,J \otimes J')$ is a cover.
\end{proof}

\begin{cor} \label{covpot}
Let $(I,J)$ be an idal cover of a finitely cocomplete tensor category $\C$. Then, for all $n,m \geq 0$ also $(I^{\otimes n},J^{\otimes m})$ is an idal cover.
\end{cor}

\begin{proof}
This follows inductively from \cref{covprod}.
\end{proof}

\section{Deligne's formula}
\label{sec:deligne}

The following result is a variant of Deligne's Formula \cite[Appendix, Proposition 4]{Har66}  (which, according to Deligne, was already well-known before him), where the ideal $J \subseteq \O_X$ is replaced by an idal $J \to \O_X$ and the submodule $J^n M \subseteq M$ by the tensor product $J^{\otimes n} \otimes M$. The proof is almost the same as in \cite[Proposition 6.9.17]{EGAI}, but it does not use the Artin--Rees Lemma because tensor products allow us more freedom to define homomorphisms. As a~consequence we do not have to assume that the scheme $X$ is noetherian.

\begin{prop} \label{deligne}
Let $X$ be a quasi-compact quasi-separated scheme. Let $e : J \to \O_X$ be a~quasi-coherent idal of finite presentation, and let $U=X_J$ be the associated open subscheme. Consider two quasi-coherent $\O_X$-modules $M,N$, where $M$ is of finite presentation. The canonical homomorphism
\[\rho : {\varinjlim}_n \Hom(J^{\otimes n} \otimes M,N) \to \Hom(M|_U,N|_U),\]
which maps a homomorphism $J^{\otimes n} \otimes M \to N$ to its restriction $M|_U \myiso (J^{\otimes n} \otimes M)|_U \to N|_U$, is an isomorphism. In particular, we obtain an isomorphism
\[{\varinjlim}_n \Hom(J^{\otimes n},N) \myiso \Gamma(U,N).\]
\end{prop}

\begin{proof}
The transition maps of the colimit are induced by the canonical morphisms of idals $J^{\otimes n} \otimes e :  J^{\otimes (n+1)} \to J^{\otimes n}$, see \cref{idal-prop}(\ref{transit}). Since $X$ is quasi-compact quasi-separated, and $\smash{{\varinjlim}_n}$ commutes with finite limits, we may reduce to the case that $X$ is affine as usual, cf.\ the proof of \cite[Proposition 6.9.17]{EGAI} and \cref{gen-qcqs}. So assume that $X$ is affine and let $I \coloneqq \im(e : J \to \O_X)$.

The homomorphism $\rho$ injective: Let $J^{\otimes n} \otimes M \to N$ be a homomorphism which vanishes on $U$. Its image is a quasi-coherent $\O_X$-module of finite type which vanishes on $U$, thus is annihilated by $I^p$ for some $p \geq 0$ by \cite[Proposition 6.8.4]{EGAI}. This means that the composition
$J^{\otimes (n+p)} \otimes M \to J^{\otimes n} \otimes M \to N$
vanishes. This proves the injectivity of $\rho$. Notice that here we only used that $J$ is of finite type.

The homomorphism $\rho$ is surjective: Let $(s_i)_{1 \leq i \leq \ell}$ be a family of global generators of $J$, and let $(m_j)_{1 \leq j \leq n}$ be a family of global generators of $M$. Let $f : M|_U \to N|_U$ be a homomorphism. Let $n_j = f(m_j|_U) \in \Gamma(U,N)$. Just as in the proof of \cite[Proposition 6.9.17]{EGAI} one finds some $h > 0$ such that for all $i,j$ the local section $e(s_i|_U)^h  \cdot n_j$ lifts to a global section of $N$. Let $d \coloneqq \ell (h-1)+1$. Then $J^{\otimes d}$ is generated by the global sections $s_{i_1,\dotsc,i_d} \coloneqq s_{i_1} \otimes \cdots \otimes s_{i_d}$ with $1 \leq i_k \leq \ell$, and by the pigeonhole principle for every such generator $e^{\otimes d}(s_{i_1,\dotsc,i_d}|_U)$ is a multiple of $e(s_i|_U)^h$ for some $i$. Hence, each $e^{\otimes d}(s_{i_1,\dotsc,i_d}|_U) \cdot n_j$ lifts to some global section $u_{i_1,\dotsc,i_d,j}$ of $N$.

On the free $\O_X$-module $F \coloneqq \O_X^{\{1,\dotsc,\ell\}^d  \times \{1,\dotsc,n\}}$ with global generators $x_{i_1,\dotsc,i_d,j}$ we define an epimorphism
\[p :  F \twoheadrightarrow J^{\otimes d} \otimes M,  ~ p(x_{i_1,\dotsc,i_d,j}) \coloneqq s_{i_1,\dotsc,i_d} \otimes m_j\]
and a homomorphism
\[q :  F \to N,  ~ q(x_{i_1,\dotsc,i_d,j}) \coloneqq u_{i_1,\dotsc,i_d,j}.\]
Then $q|_U = (e^{\otimes d}|_U \otimes f) \circ p|_U$ holds by construction. Hence, the kernels $P \coloneqq \ker(p)$ and $Q \coloneqq \ker(q)$ satisfy $P|_U \subseteq Q|_U$, and $P$ is of finite type since $J,M$ are of finite presentation. Hence, $I^k P \subseteq Q$ for some $k \geq 0$. This means that $e^{\otimes k} \otimes q : J^{\otimes k} \otimes F \to N$ vanishes when composed with $J^{\otimes k} \otimes P \to J^{\otimes k} \otimes F$. Since the sequence
\[\begin{tikzcd}[column sep=7.5ex]
J^{\otimes k} \otimes P \ar{r} & J^{\otimes k} \otimes F \ar{r}{J^{\otimes k} \otimes p} & J^{\otimes (k+d)} \otimes M \ar{r} & 0
\end{tikzcd}\]
is exact, there is a morphism $g : J^{\otimes (k+d)} \otimes M \to N$ such that $g \circ (J^{\otimes k} \otimes p) = e^{\otimes k} \otimes q$. It follows that
\begin{align*}
g|_U \circ (J^{\otimes k}|_U \otimes p|_U) & = (e^{\otimes k}|_U \otimes N|_U) \circ (J^{\otimes k}|_U \otimes q|_U)\\
& = (e^{\otimes k}|_U \otimes N|_U) \circ (J^{\otimes k}|_U \otimes e^{\otimes d}|_U \otimes f) \circ (J^{\otimes k}|_U \otimes p|_U) \\
& = (e^{\otimes (k+d)}|_U \otimes f) \circ (J^{\otimes k}|_U \otimes p|_U)
\end{align*}
and hence $g|_U = e^{\otimes (k+d)}|_U \otimes f$. But this precisely means that $g$ is a preimage of $f$ under $\rho$. Hence, $\rho$ is surjective.
\end{proof}

\begin{recol}
Recall that if $M,N$ are quasi-coherent $\O_X$-modules and $M$ is of finite presentation, then the $\O_X$-module $\HOM(M,N)$ is again quasi-coherent \cite[Proposition 9.1.1]{EGAIo}. Also recall that the pushforward functor $f_*$ associated to a quasi-compact quasi-separated morphism $f$ preserves quasi-coherent modules \cite[Proposition 6.7.1]{EGAI}.
\end{recol}

\begin{cor} \label{deligne-sheaf}
Let $X$ be a quasi-compact quasi-separated scheme. Let $J \to \O_X$ be a quasi-coherent idal of finite presentation, and let $j : X_J \hookrightarrow X$ be the associated open subscheme. Consider two quasi-coherent $\O_X$-modules $M,N$, where $M$ is of finite presentation. There is a~canonical isomorphism of quasi-coherent $\O_X$-modules
\[{\varinjlim}_n \HOM(J^{\otimes n} \otimes M,N) \myiso j_* \HOM(j^* M,j^* N).\]
In particular, there is a canonical isomorphism of quasi-coherent $\O_X$-modules
\[{\varinjlim}_n \HOM(J^{\otimes n},N) \myiso j_* j^* N.\]
\end{cor}

\begin{proof}
If $W \subseteq X$ is an affine open subscheme, then by \cref{deligne} we have canonical isomorphisms
\begin{align*}
\Gamma\bigl(W,{\varinjlim}_n \HOM(J^{\otimes n} \otimes M,N)\bigr) & \cong {\varinjlim}_n \Hom\bigl(J|_W^{\otimes n} \otimes M|_W,N|_W\bigr) \\
&  \cong \Hom\bigl(M|_{W \cap X_J},N|_{W \cap X_J}\bigr) \\
& \cong  \Gamma\bigl(W,j_* \HOM(j^* M,j^* N)\bigr). \qedhere
\end{align*}
\end{proof}

The following corollary is of independent interest. It generalizes \cite[Example 2.3]{Gro17} where the scheme $X$ is assumed to be noetherian.

\begin{cor} \label{idalgenerate}
Let $X$ be a quasi-compact quasi-separated scheme and let $X = \bigcup_i X_i$ be an open affine covering. Choose idals $J_i \to \O_X$ of finite presentation such that $X_i = X_{J_i}$. Then the $J_i$ generate $\Q(X)$ as a cocomplete tensor category. Specifically, any quasi-coherent $\O_X$-module is a quotient of a direct sum of tensor powers of the $J_i$.
\end{cor}

\begin{proof}
Let $M \in \Q(X)$. If $F$ denotes the set of all quasi-coherent submodules of $M$ of finite type, then the canonical homomorphism $\bigoplus_{N \in F} N \to M$ is an epimorphism \cite[Corollaire 6.9.9]{EGAI}. Hence, we may assume that $M$ is itself of finite type. Since each $X_i$ is affine, there is a finite set $S_i$ and an epimorphism $\O_X^{\oplus S_i} |_{X_i} \twoheadrightarrow M|_{X_i}$. By \cref{deligne} it extends to a homomorphism $(J_i^{\otimes n_i})^{\oplus S_i} \to M$ for some $n_i \geq 0$, which is thus an epimorphism when restricted to $X_i$. Thereby we obtain an epimorphism $\bigoplus_i (J_i^{\otimes n_i})^{\oplus S_i} \twoheadrightarrow M$.
\end{proof}

We also record the following interesting corollary, which foreshadows \cref{intro-1} (but is strictly weaker). It was already obtained in \cite[Lemma 5.11.14]{Bra14} with different methods.

\begin{cor} \label{boxpres}
Let $X,Y$ be two quasi-compact quasi-separated $\IK$-schemes, where $\IK$ is any commutative ring. Then every $M \in \Q(X \times_\IK Y)$ admits a presentation of the form
\[A' \boxtimes B' \to A \boxtimes B \to M \to 0,\]
for $A,A' \in \Q(X)$ and $B,B' \in \Q(Y)$, where by definition $A \boxtimes B \coloneqq p_X^*(A) \otimes p_Y^*(B)$. If~$M$ is of finite presentation, then $A,A',B,B'$ can also be chosen to be of finite presentation.
\end{cor}

\begin{proof}
Let $ X = \bigcup_i X_i$, $Y = \bigcup_j Y_j$ be open affine coverings and choose ideals $I_i \to \O_X$, $J_j \to \O_Y$ of finite presentation with $ X_i = X_{I_i}$ and $Y_j = Y_{J_j}$. Then $X \times_\IK Y = \bigcup_{i,j} (X_i \times_\IK Y_j)$ is an open affine covering, and we have $X_i \times_\IK Y_j = (X \times_\IK Y)_{I_i \boxtimes J_j}$ by \cref{Xprod}. If $M \in \Q(X \times_\IK Y)$, then by \cref{idalgenerate} there is an epimorphism of a direct sum of tensor powers of the $I_i \boxtimes J_j$ onto $M$. In particular, there is a set $S$ and $A_s \in \Q_{\fp}(X)$, $B_s \in \Q_{\fp}(Y)$ for $s \in S$ with an epimorphism $\bigoplus_{s \in S} (A_s \boxtimes B_s) \twoheadrightarrow M$. If $M$ is of finite type, $S$ can be chosen to be finite. Let $A \coloneqq \bigoplus_{s \in S} A_s$ and $B \coloneqq \bigoplus_{s \in S} B_s$. Then there is an epimorphism $A \boxtimes B \twoheadrightarrow \bigoplus_{s \in S} (A_s \boxtimes B_s)$. Hence, there is an epimorphism $A \boxtimes B \twoheadrightarrow M$. Now we apply the same procedure to its kernel (which is of finite type if $M$ is of finite presentation) to obtain the desired presentation.
\end{proof}

\section{Localizations associated to idals}    
\label{sec:tensorloc}

In this section we will find an analogue of the open subscheme $X_J \hookrightarrow X$ for cocomplete tensor categories. That is, for an idal $J$ of a well-behaved cocomplete tensor category $\C$ we are going to construct a well-behaved cocomplete tensor category $\C_J$ equipped with a cocontinuous tensor functor $\C \to \C_J$. In the case $\C=\Q(X)$ we expect that
\[\Q(X)_J \simeq \Q(X_J).\]
Also, $\C \to \C_J$ should enjoy a universal property which is similar to the one of $X_J \to X$ in \cref{ueopen}, namely that it is the universal solution to the problem of inverting $J \to \O_X$. We refer to \cite{Bra14} for plenty of examples of constructions from schemes which can be transported to cocomplete tensor categories, including their universal properties.

We start off by giving a more concrete description of $\Q(X_J)$, using Deligne's formula from \cref{sec:deligne}.

\begin{prop} \label{deligne-cat}
Let $X$ be a quasi-compact quasi-separated scheme and let $J \to \O_X$ be a~quasi-coherent idal of finite presentation. Then there is an equivalence of categories between $\Q(X_J)$ and the category of those $M \in \Q(X)$ such that the canonical homomorphism
\[M \to \HOM(J,M)\]
is an isomorphism, i.e.\ those $M$ who ``believe'' that $J \to \O_X$ is an isomorphism.
\end{prop}

\begin{proof}
Let $j : X_J \to X$ be the open immersion. Since $j^*$ is left adjoint to the fully faithful functor $j_*$, this adjunction restricts to an equivalence between $\Q(X_J)$ and the category of those $M \in \Q(X)$ such that the unit morphism $M \to j_* j^* M$ is an isomorphism. By \cref{deligne-sheaf} we have
\[j_* j^* M \cong {\varinjlim}_n \HOM(J^{\otimes n},M).\]
Now let $M \in \Q(X)$ be such that the canonical homomorphism $M \to \HOM(J,M)$ is an isomorphism. Then we obtain isomorphisms
\[M \myiso \HOM(J,M) \myiso \HOM\bigl(J,\HOM(J,M)\bigr) \myiso \HOM(J^{\otimes 2},M).\]
In fact, by induction, it follows that the canonical morphism $M \to \HOM(J^{\otimes n},M)$ is an isomorphism for all $n \in \IN$, and hence that $M \to {\varinjlim}_n \HOM(J^{\otimes n},M)$ is an isomorphism as well.

For the converse, we observe
\begin{align*}
\HOM\bigl(J,{\varinjlim}_n \HOM(J^{\otimes n},M)\bigr)&  \cong {\varinjlim}_n \HOM\bigl(J,\HOM(J^{\otimes n},M)\bigr) \\
& \cong {\varinjlim}_n \HOM(J^{\otimes (n+1)},M) \\
& \cong {\varinjlim}_n \HOM(J^{\otimes n},M).
\end{align*}
Here we have used that $J \to \O_X$ is an idal in order to ensure that the transition maps agree and that this isomorphism
\[\HOM\bigl(J,{\varinjlim}_n \HOM(J^{\otimes n},M)\bigr) \myiso {\varinjlim}_n \HOM(J^{\otimes n},M)\]
is the one induced by $J \to \O_X$.
\end{proof}

We will now transport this construction to cocomplete tensor categories. However, we prefer an abstract definition of $\C_J$ which is motivated by the universal property of $X_J$ in \cref{ueopen}, and \cref{deligne-cat} will motivate its construction. Here $\Hom_{\c\otimes}$ denotes the category of cocontinuous tensor functors.

\begin{defi}
Let $\C$ be a cocomplete tensor category and let $J \to \O_\C$ be an idal in $\C$. Then $\C_J$ is defined to be a cocomplete tensor category equipped with natural equivalences of categories
\[\Hom_{\c\otimes}(\C_J,\D) \myiso \bigl\{F \in \Hom_{\c\otimes}(\C,\D) : F(J) \to F(\O_\C) \text{ is an isomorphism}\bigr\}\]
for cocomplete tensor categories $\D$. Thus, $\C_J$ is the localization of $\C$ at $J \to \O_\C$ in the $2$-category of cocomplete tensor categories.

In more concrete terms, this means that we have a cocontinuous tensor functor
\[R_J : \C \to \C_J\]
which maps $J \to \O_\C$ to an isomorphism in $\C_J$ and has the following bicategorical universal property: If $F : \C \to \D$ is a cocontinuous tensor functor which maps $J \to \O_\C$ to an isomorphism in $\D$, then there is a cocontinuous tensor functor $\widetilde{F} : \C_J \to \D$ with $\widetilde{F} \circ R_J \cong F$. Moreover, if $G,H : \C_J \to \D$ are two cocontinuous tensor functors, then every morphism of tensor functors $G \circ R_J \to H \circ R_J$ is induced by a unique morphism of tensor functors $G \to H$.
\end{defi}

\begin{recol}
In the following we need to work with locally presentable categories \cite{AR94,GU71}, in particular to show the existence of $\C_J$. We recall that a \emph{locally presentable tensor category} is a cocomplete tensor category whose underlying category is locally presentable. These are automatically closed by Freyd's special adjoint functor theorem \cite[Remark~3.1.17]{Bra14}; the internal hom-objects will be denoted by $\HOM$. A \emph{locally finitely presentable tensor category} is a cocomplete tensor category $\C$ whose underlying category is locally finitely presentable and such that the finitely presentable objects are closed under finite tensor products \cite{Kel82}. In particular, it is required that the unit object $\O_\C$ is finitely presentable. It follows that for every finitely presentable object $I \in \C$ the functor $\HOM(I,-) : \C \to \C$ preserves filtered colimits. For example, if $X$ is a quasi-compact quasi-separated scheme, then $\Q(X)$ is a~locally finitely presentable tensor category, and the finitely presentable objects are precisely the quasi-coherent $\O_X$-modules of finite presentation in the usual sense \cite[Section 2.3]{Bra14}. 
\end{recol}

\begin{thm}\noindent \label{locex}
\begin{enu}
\item If $\C$ is a locally presentable tensor category and $J \to \O_\C$ is an idal in $\C$, then $\C_J$ exists and is again a locally presentable tensor category. Its underlying category is the reflective subcategory of $\C$ consisting of those $M \in \C$ such that the canonical morphism
\[M \to \HOM(J,M)\]
is an isomorphism. We denote the reflector by $R_J : \C \to \C_J$.
\item If $\C$ is a locally finitely presentable tensor category and $J \to \O_\C$ is a finitely presentable idal in $\C$, then $\C_J$ is also a locally finitely presentable tensor category, and $R_J : \C \to \C_J$ preserves finitely presentable objects. Moreover, we have $R_J \cong {\varinjlim}_n \HOM(J^{\otimes n},-)$ as functors.
\item If $\C=\Q(X)$ for some quasi-compact quasi-separated scheme and $J \to \O_X$ is a quasi-coherent idal of finite presentation, then
\[\Q(X)_J \simeq \Q(X_J).\]
Here, $R_J : \Q(X) \to \Q(X_J)$ is the restriction functor.
\end{enu}
\end{thm}

\begin{proof}
1.\ The general construction of tensor categorical localization was explained for any set of parallel morphisms in \cite[Theorem 5.8.12, Remark 5.8.13]{Bra14}, but has probably been known before. Explicitly, one defines the full subcategory
\[\C_J \coloneqq \bigl\{M \in \C : M \to \HOM(J,M) \text{ is an isomorphism}\bigr\}.\]
The reflector $R_J : \C \to \C_J$ is constructed in two steps, the first one ensuring that the morphism $M \to \HOM(J,M)$ becomes a monomorphism, the second one that $M \to \HOM(J,M)$ becomes an isomorphism. This special case also appeared in \cite[Example 5.8.17]{Bra14}, but the description of the reflector there is not quite correct in general. At least, we will see in 2.\ below that the description works in the setting of locally finitely presentable tensor categories.

Next, $\C_J$ becomes a locally presentable tensor category in such a way that the reflector $R_J : \C \to \C_J$ becomes a cocontinuous tensor functor. The unit object is $R_J(\O_\C)$, and the tensor product of two objects $M,N \in \C_J$ is defined by $R_J(M \otimes N)$. The colimit of a diagram in $\C_J$ is the reflection of the colimit of the underlying diagram in $\C$.

Let us briefly check the universal property. The idal $R_J(J) \to R_J(\O_\C) = \O_{\C_J}$ is an isomorphism, because for every $M \in \C_J$ the induced morphism
\[\HOM(\O_{\C_J},M) \to \HOM(R_J(J),M)\]
identifies with the isomorphism $M \to \HOM(J,M)$. If $\D$ is a cocomplete tensor category and $F : \C \to \D$ is a cocontinuous tensor functor which maps the idal $J \to \O_\C$ to an isomorphism, then the functor $F' \coloneqq F|_{\C_J} : \C_J \to \D$ is cocontinuous and admits the structure of a tensor functor with $F' \circ R_J \cong F$.

2.\ Let $M \in \C$. Using the natural maps $J^{\otimes n} \to J^{\otimes m}$ for $n \geq m$, we can define the colimit
\[R_J(M) \coloneqq {\varinjlim}_n \HOM(J^{\otimes n},M)\]
with a natural map $M \to R_J(M)$. A calculation similar to the proof of \cref{deligne-cat} shows that $R_J(M) \in \C_J$. Here, we use that $\HOM(J,-)$ preserves filtered colimits and also the idal property to ensure that the transition maps fit together. If $M \to N$ is a morphism in $\C$ with $N \in \C_J$, then $N \to R_J(N)$ is an isomorphism, so that we obtain a morphism $R_J(M) \to R_J(N) \myiso N$, which is in fact the unique morphism $R_J(M) \to N$ extending $M \to N$. Hence, $R_J : \C \to \C_J$ is, indeed, the reflector. Since $\HOM(J,-)$ preserves filtered colimits, it is clear that $\C_J$ is closed under filtered colimits in $\C$. In particular, the inclusion functor $\C_J \hookrightarrow \C$ preserves filtered colimits. Now a formal argument \cite[Exercise 1.s]{AR94} shows that its left adjoint $R_J : \C \to \C_J$ preserves finitely presentable objects. It follows that $\C_J$ is a locally finitely presentable category, and that the finitely presentable objects of $\C_J$ are precisely the retracts of those $R_J(M)$, where $M \in \C$ is a finitely presentable object. From this we deduce that the finitely presentable objects in $\C_J$ are closed under finite tensor products.

3.\ This follows from the construction in 1.\ as well as \cref{deligne-cat}.
\end{proof}

\begin{rem} \label{fcapproach}
Since $\C_J$ can also be seen as a bicategorical coinverter, its existence in the locally finitely presentable case and its description as a full subcategory of $\C$ also follow from \cite[Proposition 5.2, Lemma 5.3, Proposition 5.13]{Bra20}. This even works for arbitrary morphisms $J \to \O_\C$ with $J \in \C_{\fp}$. In this generality, the construction can be divided into two steps, the first one being a bicategorical coequifer \cite[Proposition 5.12]{Bra20} which universally turns $J \to \O_\C$ into an idal, the second one being the localization at an idal.

From the proof in loc.cit.\ we also see that actually $\C_J = \Ind((\C_{\fp})_J)$ for a localization $(\C_{\fp})_J$ of $\C_{\fp}$ in the $2$-category of essentially small finitely cocomplete tensor categories. It follows in particular $\Q_{\fp}(X)_J \simeq \Q_{\fp}(X_J)$.

This prompts the question why we do not consistently work in that more basic $2$-category, also in order to avoid switching back and forth between $\Q(X)$ and $\Q_{\fp}(X)$ as we do later. The (interconnected) reasons are the following: (1) infinite colimits are very convenient, (2) we will need finite limits in \cref{global-glue} below, which do not necessarily exist in $\Q_{\fp}(X)$, (3) usually there is no pushforward functor $\Q_{\fp}(X_J) \to \Q_{\fp}(X)$, (4) the localization of an essentially small finitely cocomplete tensor category cannot be realized as a full subcategory it, (5) probably there is no direct proof of $\Q_{\fp}(X)_J \simeq \Q_{\fp}(X_J)$ without repeating the arguments in Deligne's formula (\cref{deligne}) and thus implicitly using $\Q(X)$ anyway.
\end{rem}

\begin{rem}\label{formalprop}
Here we list some formal properties of $\C_J$.
\begin{enu}
\item 
If $\C$ is $\IK$-linear for some commutative ring $\IK$, then $\C_J$ will also be $\IK$-linear, and the universal property holds in the $2$-category of cocomplete $\IK$-linear tensor categories and cocontinuous $\IK$-linear tensor functors.
\item If $F : \C \to \D$ is a cocontinuous tensor functor, then it induces a cocontinuous tensor functor $F_J : \C_J \to \D_{F(J)}$ characterized by $F_J \circ R_J \cong R_{F(J)} \circ F$. Moreover, if $\C,\D$ are locally finitely presentable tensor categories and $F$ preserves finitely presentable objects, then the same is true for $F_J$.
\item\label{idalpull} Any morphism of idals $f : J \to J'$ in $\C$ induces a cocontinuous tensor functor $f^*: \C_{J'} \to \C_J$ characterized by $f^* \circ R_{J'} \cong R_J$. This is because $R_J$ maps $J \to J' \to \O_\C$ and hence, by \cref{isiso}, also $J' \to \O_\C$ to an isomorphism.
\end{enu}
\end{rem}
 
Next, we treat iterated localizations. The result is a categorification of the well-known result from commutative algebra.

\begin{prop} \label{twoidal}
Let $\C$ be a locally presentable tensor category. Let $I \to \O_\C \leftarrow J$ be two idals in $\C$. There is a diagram of cocontinuous tensor functors
\[\begin{tikzcd}[row sep=6ex, column sep=6ex] \C \ar{r}{R_I} \ar{d}[swap]{R_J} \ar{dr}[description]{R_{I \otimes J}} & \C_I \ar{d}{\overline{R}_J} \\ \C_J \ar{r}[swap]{\overline{R}_I} & \C_{I \otimes J} \end{tikzcd}\]
which commutes up to isomorphisms. Besides, $\smash{\overline{R}_I}$ and $\smash{\overline{R}_J}$ induce equivalences of cocomplete tensor categories
\[(\C_J)_{R_J(I)} \simeq \C_{I \otimes J} \simeq (\C_I)_{R_I(J)}.\]
The corresponding statement also holds in the linear case.
\end{prop}

\begin{proof}
It suffices to prove the following claim: a cocontinuous tensor functor $F : \C \to \D$ maps $I \otimes J \to \O_\C$ to an isomorphism if and only if $F$ maps $I \to \O_\C$ and $J \to \O_\C$ to isomorphisms. In fact, this shows that $\C_{I \otimes J}$ has the same universal property as $(\C_J)_{R_J(I)}$, and likewise $(\C_I)_{R_I(J)}$, and it also allows us to construct $\smash{\overline{R}_I}$ and $\smash{\overline{R}_J}$. The direction $\impliedby$ of the claim is trivial. For the direction $\implies\!$, assume that $F$ maps $I \otimes J \to \O_\C$ to an isomorphism. Since that morphism factors as $I \otimes J \to I \to \O_\C$, the functor maps $I \to \O_\C$ to a split epimorphism and hence to an isomorphism by \cref{isiso}. Similarly it follows that $F$ maps $J \to \O_\C$ to an isomorphism.
\end{proof}

\begin{rem} \label{expldes}
In the locally finitely presentable case, we can make the functors of \cref{twoidal} more explicit. By definition, $\smash{\overline{R}_I}$ maps an object $M \in \C_J \subseteq \C$ to $R_{I \otimes J}(M) \in \C_{I \otimes J}$, which is given by the colimit of the sequence
\[\HOM(I^{\otimes n} \otimes J^{\otimes n},M) \cong \HOM\bigl(I^{\otimes n},\HOM(J^{\otimes n},M)\bigr) \cong \HOM(I^{\otimes n},M).\]
Thus, there is a canonical isomorphism $\smash{\overline{R}_I(M) \cong R_I(M)}$. We cannot write $\smash{\overline{R}_I \cong R_I}$ since the domains are different. Similarly, we have $\smash{\overline{R}_J(M) \cong R_J(M)}$.  We also see that
\[\C_{I \otimes J} = \C_I \cap \C_J\]
as full subcategories of $\C$. This is analogous to \cref{Xprod}(\ref{Xprodinter}). 
\end{rem}

The following theorem is important because it enables us to glue objects with respect to an idal cover.

\begin{thm} \label{global-glue}
Let $\C$ be a locally finitely presentable tensor category. Let $I \to \O_\C \leftarrow J$ be a~cover consisting of two finitely presentable idals. Then the square from \cref{twoidal}
\[\begin{tikzcd}[row sep=6ex, column sep=6ex] \C \ar{r}{R_I} \ar{d}[swap]{R_J}  & \C_I \ar{d}{\overline{R}_J} \\ \C_J \ar{r}[swap]{\overline{R}_I} & \C_{I \otimes J} \end{tikzcd}\]
is a bicategorical pullback square in the $2$-category of cocomplete tensor categories. The corresponding statement also holds in the linear case.
\end{thm}

\begin{proof}
The bicategorical pullback $\smash{\overline{\C} \coloneqq \C_I \times_{\C_{I \otimes J}} \C_J}$ consists of triples $(A,B,\tau)$, where $A \in \C_I$, $B \in \C_J$ are objects and $\tau : R_J(A) \myiso R_I(B)$ is an isomorphism in $\C_{I \otimes J}$, where we use the identifications from \cref{expldes}. Now \cref{twoidal} yields a cocontinuous tensor functor
\[F : \C \to \overline{\C},\, M \mapsto (R_I(M),R_J(M),\tau_M),\]
where $\tau_M$ is the canonical isomorphism $R_J\bigl(R_I(M)\bigr) \myiso R_{I \otimes J}(M) \myiso R_I\bigl(R_J(M)\bigr)$. In order to show that $F$ is an equivalence of cocomplete tensor categories, it suffices to prove that the underlying functor is an equivalence of categories. In order to achieve this, we define a functor in the other direction (recall that $\C$ is complete by \cite[Corollary 1.28]{AR94})
\[G : \overline{\C} \to \C,\, (A,B,\tau) \mapsto A \times_{\tau : R_J(A) \myiso R_I(B)} B.\]
This pullback should be thought of as a gluing of $A$ and $B$ along $\tau$, and this is literally true when $\C=\Q(X)$ for some quasi-compact quasi-separated scheme $X$.

We will now prove $G \circ F \cong \id_{\C}$. If $M \in \C$, then there is a canonical morphism
\[M \to R_I(M) \times_{R_{I \otimes J}(M)} R_J(M) \myiso G(F(M)).\]
For every $n \in \IN$ we know that $I^{\otimes n}$ and $J^{\otimes n}$ form an idal cover by \cref{covpot}. Hence, for every object $T \in \C$ the square
\[\begin{tikzcd}[row sep=6ex, column sep=6ex] T \otimes I^{\otimes n} \otimes J^{\otimes n}  \ar{r}  \ar{d} & T \otimes J^{\otimes n} \ar{d} \\ T \otimes I ^{\otimes n} \ar{r} & T \end{tikzcd}\]
is a pushout square, from which it follows that
\[\begin{tikzcd}[row sep=6ex, column sep=6ex] M \ar{r} \ar{d} &  \HOM(J^{\otimes n},M) \ar{d} \\  \HOM(I^{\otimes n},M) \ar{r} & \HOM(I^{\otimes n} \otimes J^{\otimes n},M) \end{tikzcd}\]
is a pullback square. Since filtered colimits are exact in $\C$ by \cite[Proposition 1.59]{AR94}, this implies
\[M \cong {\varinjlim}_n \bigl(\HOM(I^{\otimes n},M) \times_{\HOM(I^{\otimes n} \otimes J^{\otimes n},M)} \HOM(J^{\otimes n},M)\bigr) \cong R_I(M) \times_{R_{I \otimes J}(M)} R_J(M).\]
One checks that this isomorphism is the canonical one.

Finally, we will prove $F \circ G \cong \id_{\overline{\C}}$. Let $(A,B,\tau : R_J(A) \myiso R_I(B))$ be an object in $\smash{\overline{\C}}$. Since each $\HOM(I^{\otimes n},-)$ preserves limits, in particular pullbacks, and filtered colimits are exact in~$\C$, we see that 
\[R_I(G(A,B,\tau)) \cong G(R_I(A),R_I(B),R_I(\tau)) \cong  A  \times_{\tau : R_J(A) \myiso R_I(B)} R_I(B) \cong A.\]
A similar argument shows $R_J(G(A,B,\tau)) \cong B$. These isomorphisms induce an isomorphism $F(G(A,B,\tau)) \myiso (A,B,\tau)$.
\end{proof}

Now let us check that some properties of objects in tensor categories can be tested locally with respect to an idal cover.

\begin{prop} \label{fplok}
In the situation of \cref{global-glue} an object $M \in \C$ is finitely presentable (resp.\ dualizable, resp.\ invertible, resp.\ symtrivial) if and only if the objects $R_I(M) \in \C_I$ and $R_J(M) \in \C_J$ have this property.
\end{prop}

\begin{proof}
1. If $M \in \C$ is finitely presentable, then by \cref{locex} the objects $R_I(M) \in \C_I$ and $R_J(M) \in \C_J$ are finitely presentable. The converse follows from the simple observation that $\smash{(A,B,\tau) \in \overline{\C} \coloneqq \C_I \times_{\C_{I \otimes J}} \C_J} \simeq \C$ is finitely presentable when $A \in \C_I$ and $B \in \C_J$ are finitely presentable.

2. If $M \in \C$ is dualizable, then the objects $R_I(M) \in \C_I$ and $R_J(M) \in \C_J$ are dualizable simply because $R_I$ and $R_J$ are tensor functors. Conversely, assume that $\smash{(A,B,\tau) \in \overline{\C}}$ is an object such that $A$ and $B$ are dualizable. We choose duals $A^*,B^*$ and dualize the isomorphism $\smash{\tau^{-1} : \overline{R}_I(B) \to \overline{R}_J(A)}$ to an isomorphism
\[\begin{tikzcd}[column sep=7ex]
\sigma : \overline{R}_J(A^*) = \overline{R}_J(A)^* \ar{r}{\left(\tau^{-1}\right)^*} & \overline{R}_I(B)^* = \overline{R}_I(B^*).
\end{tikzcd}\]
This defines an object $(A^*,B^*,\sigma)$ which is easily checked to be dual to $(A,B,\tau)$.

3. If $M \in \C$ is invertible, then the objects $R_I(M) \in \C_I$ and $R_J(M) \in \C_J$ are invertible simply because $R_I$ and $R_J$ are tensor functors. Conversely, assume that $\smash{(A,B,\tau) \in \overline{\C}}$ is an object such that $A$ and $B$ are invertible. Then $A$ and $B$ are dualizable, so that by 2.\ also $(A,B,\tau)$ is dualizable. The evaluation $\smash{(A,B,\tau)^* \otimes (A,B,\tau) \to \O_{\overline{\C}}}$ is an isomorphism because the evaluations $A^* \otimes A \to \O_{\C_I}$ and $B^* \otimes B \to \O_{\C_J}$ are isomorphisms. Hence, $(A,B,\tau)$ is invertible.

4. The statement about symtrivial objects follows directly from \cref{global-glue}.
\end{proof}

For the sake of completeness, we include a criterion when two localizations $\C_I$, $\C_J$ agree.

\begin{prop}
Let $\C$ be a locally finitely presentable tensor category and let $I \to \O_\C \leftarrow J$ be two finitely presentable idals in $\C$. There is a cocontinuous tensor functor $F : \C_I \to \C_J$ with $F \circ R_I \cong R_J$ if and only if there is a natural number $n$ and a morphism of idals $J^{\otimes n} \to I$. Hence, there is an equivalence of tensor categories $F : \C_I \to \C_J$ with $F \circ R_I \cong R_J$ if and only if there are natural numbers $n,m$ and morphisms of idals $J^{\otimes n} \to I$ and $I^{\otimes m} \to J$. The corresponding statements also hold in the linear case.
\end{prop}

\begin{proof}
A morphism of idals $J^{\otimes n} \to I$ induces by \cref{formalprop}(\ref{idalpull}) a cocontinuous tensor functor $\C_I \to \C_{J^{\otimes n}}$ under~$\C$, i.e.\ the evident triangle commutes up to isomorphism. Since $R_J$ maps $J \to \O_\C$ and hence also $J^{\otimes n} \to \O_\C^{\otimes n} \myiso \O_\C$ to an isomorphism, there is a cocontinuous tensor functor $\C_{J^{\otimes n}} \to \C_J$ under~$\C$. The composition is a cocontinuous tensor functor $\C_I \to \C_J$ under~$\C$. If there is also a morphism of idals $I^{\otimes m} \to J$, the induced cocontinuous tensor functor $\C_J \to \C_I$ under $\C$ is pseudo-inverse to $\C_I \to \C_J$ because this is the case on $\C$.

Conversely, assume that there is a cocontinuous tensor functor $\C_I \to \C_J$ under $\C$. This means that $R_J : \C \to \C_J$ maps $I \to \O_\C$ to an isomorphism, i.e.\ by \cref{locex} that
\[{\varinjlim}_n \HOM(J^{\otimes n},I) \to {\varinjlim}_n \HOM(J^{\otimes n},\O_\C)\]
is an isomorphism in $\C$. Since $\O_\C$ is finitely presentable, we may apply the ``global section'' functor $\Hom(\O_\C,-)$ and derive that
\[{\varinjlim}_n \Hom(J^{\otimes n},I) \to {\varinjlim}_n \Hom(J^{\otimes n},\O_\C)\]
is an isomorphism of sets (or $\IK$-modules in the $\IK$-linear case). In particular, the identity $\id_{\O_\C} : J^{\otimes 0} \to \O_\C$ has a preimage, say $J^{\otimes k} \to I$ for some natural number $k$. This means that there is some natural number $n \geq k$ such that the composition $J^{\otimes n} \to J^{\otimes k} \to I \to \O_\C$ is equal to the idal $J^{\otimes n} \to \O_\C$. Therefore, $J^{\otimes n} \to I$ is a morphism of idals.
\end{proof}

\begin{cor}
If $\C$ is a locally finitely presentable linear tensor category and $e : J \to \O_\C$ is a finitely presentable idal in $\C$, then $\C_J=0$ holds if and only if $e$ is nilpotent, i.e.\ there is some $n \in \IN$ such that $e^{\otimes n} = 0$. \hfill $\square$
\end{cor}

\begin{rem}\label{closedcat}
The tensor functor $R_I : \C \to \C_I$ is a tensor categorical analog of an open immersion. There is also a tensor categorical analog of a closed immersion, which is actually more elementary and can be described as follows, cf.\ \cite[Corollary 5.3.9]{Bra14}. If $\C$ is a cocomplete linear tensor category and $e : I \to \O_\C$ is any morphism in $\C$, then we may define $\C/I$ as a cocomplete linear tensor category equipped with a universal cocontinuous linear tensor functor $\C \to \C/I$ which maps $e$ to zero. We may construct the underlying category of $\C/I$ as the reflective subcategory of $\C$ consisting of those $M \in \C$ such that $M \otimes e = 0$. The functor $\C \to \C/I$, $M \mapsto M \otimes \O_\C/I$ provides a reflection, where $\O_\C/I$ denotes the cokernel of $I \to \O_\C$. The unit of $\C/I$ is $\O_\C/I$. Binary tensor products and arbitrary colimits in $\C/I$ are directly inherited from $\C$. If $\C=\Q(X)$ for some scheme $X$ and $I \to \O_X$ is a morphism, then the construction before shows $\C/I \simeq \Q(V(I))$, where $V(I) \coloneqq \Spec(\O_X/I)$ is the closed subscheme associated to $I$.
\end{rem}

\section{Fiber products of schemes}
\label{sec:fiber}

The theory from \cref{sec:tensorloc} can be used to derive our main theorems about (fiber) products of schemes.

\begin{recol}
Recall that the bicategorical coproduct of two objects $A,B$ of a $2$-category (or even a bicategory) is a pair of morphisms $A \rightarrow A \sqcup B \leftarrow B$ which induces an equivalence of categories $\Hom(A \sqcup B,T) \myiso \Hom(A,T) \times \Hom(B,T)$ for every object $T$.
\end{recol}

\begin{thm}  \label{prodthm}
Let $\IK$ be a commutative ring and let $X,Y$ be two quasi-compact quasi-separated $\IK$-schemes. Then the cocontinuous $\IK$-linear tensor functors
\[\Q(X) \xrightarrow{p_X^*} \Q(X \times_{\IK} Y) \xleftarrow{p_Y^*} \Q(Y)\]
exhibit $\Q(X \times_{\IK} Y)$ as the bicategorical coproduct of $\Q(X)$ and $\Q(Y)$ in the $2$-category of cocomplete $\IK$-linear tensor categories and cocontinuous $\IK$-linear tensor functors.
\end{thm}

It follows that $\Q(X \times_{\IK} Y)$ is also the bicategorical coproduct of $\Q(X)$ and $\Q(Y)$ in the $2$-category of locally (finitely) presentable $\IK$-linear tensor categories and cocontinuous $\IK$-linear tensor functors. \cref{prodthm} (as well as its generalization to fiber products, see \cref{fiberprodthm} below) has been obtained in several cases before: \cite[Theorem 5.11.8]{Bra14} deals with the case that $X$ is quasi-projective, \cite[Theorem 4.2]{Sch18} deals with the case that $X,Y$ are quasi-compact semi-separated algebraic stacks with the resolution property, and \cite[Theorem 1.2]{BFN10} deals with the case that $X,Y$ are perfect (derived) stacks in the $\infty$-categorical setting; see also \cite[Corollary 9.4.2.3]{Lur18} for a generalization. Basically, the resolution property allows us to construct a \emph{descent algebra} which reduces the problem immediately to the affine case, cf.\ the proof of \cite[Theorem 5.11.17]{Bra14}, and perfect stacks are defined by a derived version of the resolution property \cite[Definition 3.2]{BFN10}. Using idals and idal covers, a reduction to the affine case is even possible when there are not enough locally free modules.
 
\begin{proof}[Proof of \cref{prodthm}]
We denote by $\Hom_{\c\otimes/ \IK}$ the category of cocontinuous $\IK$-linear tensor functors. For commutative $\IK$-algebras $A$ and cocomplete $\IK$-linear tensor categories $\C$ we have the bicategorical adjunction (see \cite[Proposition 2.2.3]{BC14} or \cite[Example 5.2.1]{Bra14})
\[\Hom_{\c\otimes/ \IK}\bigl(\Q(\Spec(A)),\C\bigr) \simeq \Hom_{\c\otimes/ \IK}\bigl(\Mod(A),\C\bigr) \simeq \Hom_{\CAlg_{\IK}}\bigl(A,\End_\C(\O_\C)\bigr).\]
It follows immediately that the claim is true if $X$ and $Y$ are affine. In the general case, by applying \cref{gen-qcqs} in each variable it suffices to prove the following: If $X = X_1 \cup X_2$ is covered by quasi-compact open subschemes $X_1$ and $X_2$ such that the theorem holds for the pairs $(X_1,Y)$, $(X_2,Y)$ and $(X_1 \cap X_2,Y)$, then it also holds for the pair $(X,Y)$. We will now omit $\IK$ from the notation. By \cite[Proposition 3.4]{Bra20} the bicategorical coproduct
\[\C \coloneqq \Q(X) \coboxtimes \Q(Y)\]
exists in the $2$-category of cocomplete linear tensor categories. Even more is true, $\C$ is a locally finitely presentable linear tensor category, and the two canonical cocontinuous tensor functors $\Q(X) \to \C \leftarrow \Q(Y)$ preserve finitely presentable objects.

By \cref{idals-qcqs} we find quasi-coherent idals $I_1 \to \O_X \leftarrow I_2$ of finite presentation such that $X_1 = X_{I_1}$ and $X_2 = X_{I_2}$, so that these form an idal cover. Consider the image $J_1 \to \O_\C \leftarrow J_2$ in $\C$ under $\Q(X) \to \C$, which is again an idal cover consisting of finitely presentable idals. By simply comparing the universal properties, we get
\[\smash{\C_{J_1} \simeq \Q(X)_{I_1} \coboxtimes \Q(Y)},\]
and therefore
\[\C_{J_1} \simeq \Q(X_{I_1}) \coboxtimes \Q(Y) = \Q(X_1) \coboxtimes \Q(Y) \simeq \Q(X_1 \times Y).\]
Similarly, we get equivalences $\C_{J_2} \simeq \Q(X_2 \times Y)$ and $\C_{J_1 \otimes J_2} \simeq \Q(X_{1,2} \times Y)$, where we abbreviate $X_{1,2} \coloneqq X_1 \cap X_2$. Under these equivalences, the canonical functor $\C_{J_i} \to \C_{J_1 \otimes J_2}$ from \cref{twoidal} corresponds to the restriction functor $\Q(X_i \times Y) \to \Q(X_{1,2} \times Y)$. Since $X \times Y$ is covered by the open subschemes $X_1 \times Y$ and $X_2 \times Y$ whose intersection is $X_{1,2} \times Y$, we deduce from \cref{global-glue} that
\[\C \simeq \C_{J_1} \times_{\C_{J_1 \otimes J_2}} \C_{J_2}\simeq \Q(X_1 \times Y) \times_{\Q(X_{1,2} \times Y)} \Q(X_2 \times Y) \simeq \Q(X \times Y).\]
This finishes the proof.
\end{proof}

\begin{rem} \label{kellyrel}
By construction of the bicategorical coproduct in \cite[Proposition 3.4]{Bra20}, we have
\[\Q(X \times_{\IK} Y) \simeq \Q(X) \coboxtimes_{\IK} \Q(Y) \simeq \Ind\bigl(\Q_{\fp}(X) \boxtimes_{\IK} \Q_{\fp}(Y)\bigr),\]
where $\boxtimes_{\IK}$ denotes Kelly's tensor product of essentially small finitely cocomplete $\IK$-linear categories (see \cite[Section 6.5]{Kel05} or \cite[Theorem 7]{LF13}). It follows
\[\Q_{\fp}(X \times_\IK Y) \simeq \Q_{\fp}(X) \boxtimes_{\IK} \Q_{\fp}(Y).\]
\end{rem}

Now let us treat fiber products of schemes.

\begin{recol}
Recall that a bicategorical pushout of two morphisms $f : C \to A$, $g : C \to B$ in a $2$-category (or even a bicategory) is a tuple $(P,i_A,i_B,\alpha)$ consisting of an object $P = A \sqcup_C B$, two $1$-morphisms $i_A : A \to P$, $i_B : B \to P$ and a $2$-isomorphism $\alpha : i_A \circ f \to i_B \circ g$,
\[\begin{tikzcd}[sep=7ex]
C \ar{r}{f} \ar{d}[swap]{g} & A \ar{d}{i_A} \ar[double equal sign distance,shorten >=1.5ex,shorten <=1.5ex, >=stealth]{dl}{\sim} \ar[phantom]{dl}[above]{\scriptstyle\alpha~~~} \\ B \ar{r}[swap]{i_B}  & P
\end{tikzcd}\]
such that for every object $T$ the induced functor
\[\Hom(P,T) \to \Hom(A,T) \times_{\Hom(C,T)} \Hom(B,T), \, h \mapsto (h \circ \iota_A, h \circ \iota_B, h \circ \alpha)\]
is an equivalence of categories. For cocomplete tensor categories we write $\smash{P = A \coboxtimes_C B}$.
\end{recol}

\begin{rem} \label{locbasechange}
If $F : \C \to \D$ is a cocontinuous tensor functor and $J \to \O_\C$ is an idal in $\C$, then the square (assuming that $\C_J$ and $\D_{F(J)}$ exist)
\[\begin{tikzcd}[row sep=6ex, column sep=6ex]
\C \ar{d}[swap]{F} \ar{r}{R_J} & \C_J \ar{d}{F_J} \\ \D \ar{r}[swap]{R_{F(J)}} & \D_{F(J)}
\end{tikzcd}\]
from \cref{formalprop} is a bicategorical pushout square in the $2$-category of cocomplete tensor categories. This follows immediately from the universal properties of $\C_J$ and $\D_{F(J)}$.

Similarly, if $\C$ is linear, $J \to \O_\C$ is any morphism and $\C \to \C/J$ is the quotient from \cref{closedcat}, then the square
\[\begin{tikzcd}[row sep=6ex, column sep=6ex]
\C \ar{d} \ar{r} & \C/J \ar{d} \\ \D \ar{r} & \D/F(J)
\end{tikzcd}\]
is a bicategorical pushout square, which follows immediately from the universal properties of $\C/J$ and $\D/F(J)$.
\end{rem}
 
\begin{thm}  \label{fiberprodthm}
Let $\IK$ be a commutative ring and let $X,Y$ be two quasi-compact quasi-separated schemes over some quasi-compact quasi-separated $\IK$-scheme $S$. Then the square
\[\begin{tikzcd}[row sep=6ex, column sep=4ex]
\Q(S) \ar{d} \ar{r} & \Q(X)  \ar{d} \\ \Q(Y) \ar{r} & \Q(X \times_S Y)
\end{tikzcd}\]
is a bicategorical pushout in the $2$-category of cocomplete $\IK$-linear tensor categories and cocontinuous $\IK$-linear tensor functors.
\end{thm}

We will give two proofs of this theorem. The first is a geometric reduction to the case of products (\cref{prodthm}) and is very much inspired by \cite[Section 4]{Sch18}. The second one is similar to our proof of \cref{prodthm}.
  
\begin{proof}[First proof of \cref{fiberprodthm}]
In the following, all schemes are understood to be quasi-compact quasi-separated. We will also omit the base ring $\IK$ from the notation. Let us say that a~morphism $X \to S$ of schemes is \emph{good} if the theorem holds for all morphisms $Y \to S$. (This property was called \emph{tensorial base change} in \cite[Section 5.11]{Bra14}.) It is easy to check that every isomorphism is good and that good morphisms are closed under composition, using $X' \times_S Y \cong X' \times_X (X \times_S Y)$ for morphisms $X' \to X \to S$ and the corresponding dual result for bicategorical pushouts.

We claim that immersions are good. This was proven in \cite[Theorem 5.11.5]{Bra14}, but what follows is another proof. If $X \to S$ is an open immersion, we have $X \cong S_I$ for some quasi-coherent idal $I \to \O_S$ of finite presentation (\cref{idals-qcqs}). Let $J \to \O_Y$ be its image under $\Q(S) \to \Q(Y)$. Then by \cref{locbasechange} we have
\[\Q(Y)_J \simeq \Q(S)_I \, \coboxtimes_{\Q(S)} \, \Q(Y).\]
Since we also have  $\Q(Y)_J \simeq \Q(Y_J) \simeq \Q(S_I \times_S Y) \simeq \Q(X \times_S Y)$ as well as $\Q(S)_I \simeq \Q(S_I) \simeq \Q(X)$, it follows that $X \to S$ is good. A similar argument works for closed immersions, using \cref{closedcat}: If $X \to S$ is a closed immersion, we have $X \cong V(I)$ for some quasi-coherent ideal $I \subseteq \O_S$. Let $J \to \O_Y$ be its image under $\Q(S) \to \Q(Y)$. Then by \cref{locbasechange} we have
\[\Q(Y)/J \simeq \Q(S)/I  \, \coboxtimes_{\Q(S)} \, \Q(Y).\]
Since we also have $\Q(Y)/J \simeq \Q(V(J)) \simeq \Q(V(I) \times_S Y) \simeq \Q(X \times_S Y)$ and $\Q(S)/I \simeq \Q(V(I)) \simeq \Q(X)$, it follows that $X \to S$ is good. Since immersions are compositions of closed and open immersions, they are good as well.

Now the general case of a morphism $X \to S$ will follow from the following fiber product diagram (with the obvious morphisms).
\[\begin{tikzcd}[row sep=6ex]
X \times_S Y \ar{r} \ar{d} & X \times Y \ar{d} \\ 
S \ar{r}[swap]{\Delta_S} & S \times S
\end{tikzcd}\]
In fact, $\Delta_S$ is an immersion, so that we already know that the square
\[\begin{tikzcd}[row sep=6ex]
\Q(S \times S) \ar{d} \ar{r} & \Q(X \times Y)  \ar{d} \\ \Q(S) \ar{r} & \Q(X \times_S Y)
\end{tikzcd}\]
is a bicategorical pushout square. By \cref{prodthm} it identifies with the following square (with the obvious tensor functors).
\[\begin{tikzcd}[row sep=6ex]
\Q(S) \coboxtimes \Q(S) \ar{d} \ar{r} & \Q(X) \coboxtimes \Q(Y)  \ar{d} \\ \Q(S) \ar{r} & \Q(X \times_S Y)
\end{tikzcd}\]
Now we are done because of the following fact: If $\A \leftarrow \C \rightarrow \B$ are morphisms in a $2$-category, then a bicategorical pushout of $\C \leftarrow \C \sqcup \C \rightarrow \A \sqcup \B$, provided that these bicategorical coproducts exist, is also a bicategorical pushout of $\A \leftarrow \C \rightarrow \B$.
\end{proof}

\begin{proof}[Second proof of \cref{fiberprodthm}]
All schemes here are understood to be quasi-compact quasi-separated. By \cite[Proposition 5.1]{Bra20} the bicategorical pushout
\[\C \coloneqq \Q(X) \coboxtimes_{\Q(S)} \Q(Y)\]
exists in the $2$-category of cocomplete $\IK$-linear tensor categories. Moreover, it is a locally finitely presentable $\IK$-linear tensor category, and the three tensor functors from $\Q(X)$, $\Q(Y)$ and $\Q(S)$ preserve finitely presentable objects. There is a canonical cocontinuous $\IK$-linear tensor functor $\C \to \Q(X \times_S Y)$, which we claim to be an equivalence. If $S$ is affine, we are done by \cref{prodthm}. In the general case, by \cref{gen-qcqs} we may assume that $S = S_1 \cup S_2$ for quasi-compact open subschemes $S_1$ and $S_2$ such that the theorem is true for schemes over $S_1$, $S_2$ and $S_{1,2} \coloneqq S_1 \cap S_2$. By \cref{idals-qcqs} we find quasi-coherent idals $I_i \to \O_S$ of finite presentation such that $S_i = S_{I_i}$. Let $J_i \to \O_X$ be the image in $\Q(X)$ and let $K_i \to \O_Y$ be the image in $\Q(Y)$. Also, let $L_i \to \O_\C$ be the image in $\C$. All these are finitely presentable idal covers. By simply comparing the universal properties and then applying the hypothesis to $S_1$, we get
\[\C_{L_1} \simeq \Q(X)_{J_1} \coboxtimes_{\Q(S)_{I_1}} \Q(Y)_{K_1} \simeq \Q(X_{J_1} \times_{S_{I_1}} Y_{K_1}).\]
The same holds for $I_2$ and $I_{12} \coloneqq I_1 \otimes I_2$. Now \cref{global-glue} yields
\begin{align*}
\C  & \simeq \C_{L_1} \times_{\C_{L_1 \otimes L_2}} \C_{L_2} \simeq \Q(X_{J_1} \times_{S_{I_1}} Y_{K_1}) \times_{\Q(X_{J_{12}} \times_{S_{I_{12}}} Y_{K_{12}})} \Q(X_{J_2} \times_{S_{I_2}} Y_{K_2}) \\
& \simeq \Q(X \times_S Y). \qedhere
\end{align*}
\end{proof}

\begin{rem}
Because of the construction of the pushout in \cite[Proposition 5.1]{Bra20}, we can derive that $\Q_{\fp}(X \times_S Y)$ is the bicategorical pushout of $\Q_{\fp}(X)$ and $\Q_{\fp}(Y)$ over $\Q_{\fp}(S)$ in the $2$-category of essentially small finitely cocomplete $\IK$-linear tensor categories.
\end{rem}

\begin{rem}
Our proof of \cref{prodthm} uses the existence of bicategorical coproducts, and our second proof of \cref{fiberprodthm} uses the existence of bicategorical pushouts. In the next section, we will give alternative proofs which do not rely on these non-trivial category theoretic results.
\end{rem}

Finally, we present a non-symmetric generalization of \cref{prodthm}, where we replace one of the schemes by a locally finitely presentable tensor category. For this we introduce the following definition of internal quasi-coherent modules.
 
\begin{defi}
Let $X$ be a $\IK$-scheme and let $\C$ be a cocomplete $\IK$-linear category. We can view $X$ as a functor $X : \CAlg_\IK \to \Set$, in particular as a pseudo-functor $X : \CAlg_{\IK} \to \Cat$. The following definition works for every such pseudo-functor. We define another pseudo-functor $\Mod_{\C} : \CAlg_\IK \to \Cat$ by mapping an algebra $R$ to the category $\Mod_{\C}(R)$ of left $R$-modules in $\C$, i.e.\ objects in $\C$ with a left $R$-action, and a homomorphism $R \to S$ to the functor $S \otimes_R - : \Mod_{\C}(R) \to \Mod_{\C}(S)$.

We define $\Q_{\C}(X)$ as the category of pseudo-natural transformations $X \to \Mod_{\C}$, which we call \emph{quasi-coherent $\O_X$-modules internal to~$\C$}. Explicitly, such a module $M$ associates to every $R$-valued point $p \in X(R)$ a left $R$-module $M_p \in \Mod_{\C}(R)$ and to every homomorphism of algebras $\varphi : R \to S$ and every $p \in X(R)$ an isomorphism of $S$-modules $S \otimes_R M_p \myiso M_{X(\varphi)(p)}$ subject to two evident coherence conditions.  The notion of morphisms is obvious. Clearly, $\Q_{\C}(X)$ is a cocomplete $\IK$-linear category.

Moreover, if $\C$ is a cocomplete $\IK$-linear tensor category, then the same is true for $\Q_{\C}(X)$, since it is true for each $\Mod_{\C}(R)$ by using a suitably defined tensor product $\otimes_R$. Notice that every morphism of $\IK$-schemes $f : X \to Y$ induces a cocontinuous $\IK$-linear tensor functor $f^* : \Q_{\C}(Y) \to \Q_{\C}(X)$ which is simply defined by $(f^* M)_p \coloneqq M_{f(p)}$.
\end{defi}

\begin{rem}
Let $X$ be a $\IK$-scheme. Then the ``functorial'' characterization of quasi-coherent $\O_X$-modules yields $\Q_{\Mod(\IK)}(X) \simeq \Q(X)$, which justifies our notation. More generally, for $\IK$-schemes $Y$ (or even algebraic stacks) we have
\[\Q_{\Q(Y)}(X) \simeq \Q(X \times_{\IK} Y).\]
The latter example shows that $X \mapsto \Q_{\Q(Y)}(X)$ is a stack (fibered in categories, not in groupoids) with respect to the Zariski topology. This holds in a much more general setting, see \cref{myStack} below. For a more interesting example, consider $\C=\Ch(\IK)$, the cocomplete tensor category of chain complexes of $\IK$-modules. Then $\Q_{\Ch(\IK)}(X) \simeq \Ch(\Q(X))$ is the cocomplete tensor category of chain complexes of quasi-coherent $\O_X$-modules.
\end{rem}

\begin{thm} \label{myStack}
Let $\C$ be a locally finitely presentable $\IK$-linear tensor category. Then the pseudo-functor $\Sch_{\IK}^{\op} \to \Cat$, $X \mapsto \Q_{\C}(X)$, $f \mapsto f^*$ is a stack with respect to the Zariski topology.
\end{thm}

\begin{proof}
Since we have defined quasi-coherent $\O_X$-modules internal to $\C$ essentially by their restrictions to affine schemes, and every open covering of an affine scheme can be refined by a~finite open covering of basic-open subsets, it suffices to prove the stack property, i.e.\ descent for open coverings of the form $\Spec(A) = \bigcup_{i=1}^{s} D(f_i)$, where $A$ is a commutative \mbox{$\IK$-algebra} and $f_1,\dotsc,f_s \in A$ are elements with $\langle f_1,\dotsc,f_s \rangle = A$. That is, we have to prove that $\Mod_{\C}(A)$ is equivalent to the category of $M_i \in \Mod_{\C}(A[f_i^{-1}])$ for $i=1,\dotsc,s$ with isomorphisms of $A[f_i^{-1},f_j^{-1}]$-modules
\[A[f_i^{-1},f_j^{-1}] \otimes_{A_i} M_i \myiso A[f_i^{-1},f_j^{-1}] \otimes_{A_j} M_j\]
internal to $\C$ satisfying the cocycle condition. A direct proof is possible via a categorification of the well-known case $\C=\Mod(\IK)$. Alternatively, we can use tensor categorical descent theory as developed in \cite[Section 4.10]{Bra14}. Consider the commutative $A$-algebra $A' \coloneqq \prod_{i=1}^{s} A[f_i^{-1}]$. Then, the claim is equivalent to the assertion that $A' \otimes_{\IK} \O_\C$ is a \emph{descent algebra} in $\Mod_{\C}(A)$ as defined in \cite[Definition 4.10.1]{Bra14}. It follows from \cite[Proposition 4.10.9]{Bra14} applied to the smooth surjection $\Spec(A') \to \Spec(A)$ that $A'$ is a \emph{special descent algebra} in $\Mod(A)$ in the sense of \cite[Definition 4.10.6]{Bra14}. Alternatively, one can check this directly by writing $A'$ as the filtered colimit of the sequence of $A$-modules $A^s \to A^s \to A^s \to \cdots$ with the transition maps $(f_1,\dotsc,f_s) : A^s \to A^s$. Special descent algebras are preserved by every cocontinuous \mbox{$\IK$-linear} tensor functor, which is their main advantage as opposed to general descent algebras. We apply this to $\Mod(A) \to \Mod_{\C}(A)$, $M \mapsto M \otimes_{\IK} \O_\C$ and conclude that $A' \otimes_{\IK} \O_\C$ is a special descent algebra in $\Mod_{\C}(A)$. By \cite[Proposition 4.10.8]{Bra14} it is therefore a descent algebra as well.
\end{proof}
 
This enables us to prove the following generalization of \cref{prodthm}.
 
\begin{thm}\label{genprodthm}
Let $X$ be a quasi-compact quasi-separated $\IK$-scheme and let $\C$ be a locally finitely presentable $\IK$-linear tensor category. Then $\Q_{\C}(X)$ is a bicategorical coproduct of $\Q(X)$ and $\C$ in the $2$-category of cocomplete $\IK$-linear tensor categories, so that
\[\Q_{\C}(X) \simeq \Q(X) \coboxtimes_{\IK} \C.\]
\end{thm}

\begin{proof}
There are two canonical cocontinuous $\IK$-linear tensor functors $\Q(X) \to \Q_{\C}(X)$, $M \mapsto (M_p \otimes_{\IK} \O_\C)_{p \in X(R)}$ and $\C \to \Q_{\C}(X)$, $T \mapsto (R \otimes_{\IK} T)_{p \in X(R)}$. If $X$ is affine, say $X = \Spec(A)$, the $2$-categorical Yoneda Lemma implies $\Q_{\C}(X) \simeq \Mod_{\C}(A)$. It follows from \cite[Proposition 5.3.1]{Bra14} applied to the commutative algebra object $A \otimes_{\IK} \O_\C$ in~$\C$ that $\Mod_{\C}(A)$ is the bicategorical coproduct of $\Mod(A) \simeq \Q(X)$ and $\C$. For the general case, we may assume $X = X_1 \cup X_2$ for two quasi-compact open subschemes $X_1,X_2$ such that the claim is true for $X_1$, $X_2$ and $X_1 \cap X_2$. Choose finitely presentable idals $I_i \to \O_X$ with $X_i = X_{I_i}$. By \cite[Proposition 3.4]{Bra20} the bicategorical coproduct $\smash{\D \coloneqq \Q(X) \coboxtimes_{\IK} \C}$ exists in the $2$-category of cocomplete $\IK$-linear tensor categories. Moreover, $\D$ is a locally finitely presentable $\IK$-linear tensor category and $\Q(X) \to \D$ preserves finitely presentable objects. In particular, the images of $I_i$ provide a finitely presentable idal cover $J_i$ of $\D$. We have
\[\D_{J_i} \simeq \Q(X)_{I_i} \coboxtimes_{\IK} \C \simeq \Q(X_i) \coboxtimes_{\IK} \C \simeq \Q_{\C}(X_i),\]
similarly for the intersection. Hence, \cref{global-glue} implies
\[\D  \simeq \D_{J_1} \times_{\D_{J_1 \otimes J_2}} \D_{J_2} \simeq \Q_{\C}(X_1) \times_{\Q_{\C}(X_1 \cap X_2)} \Q_{\C}(X_2).\]
From \cref{myStack} we deduce
\[\Q_{\C}(X_1) \times_{\Q_{\C}(X_1 \cap X_2)} \Q_{\C}(X_2)\simeq \Q_{\C}(X).\qedhere\]
\end{proof}

\begin{rem} \label{stacksprod}
Rydh has shown in \cite{Ryd15} that many algebraic stacks $Y$ have the \emph{completeness property}, meaning that every quasi-coherent $\O_Y$-module is a filtered colimit of quasi-coherent $\O_Y$-modules of finite presentation. This implies that $\Q(Y)$ is a locally finitely presentable tensor category (and that the quasi-coherent $\O_Y$-modules of finite presentation are exactly those which are finitely presentable in the categorical sense). For example, noetherian and quasi-compact quasi-separated Deligne-Mumford stacks have the completeness property. In an unpublished work Rydh has shown that, in fact, every quasi-compact quasi-separated algebraic stack has the completeness property \cite{Ryd16}. Therefore \cref{genprodthm} implies that
\[\smash{\Q(X \times_\IK Y) \simeq \Q(X) \coboxtimes_\IK \Q(Y)}\]
holds for quasi-compact quasi-separated schemes $X$ and quasi-compact quasi-separated algebraic stacks $Y$. We conjecture that this remains valid if $X$ is a quasi-compact quasi-separated algebraic stack. One possible approach to prove this conjecture is a tensor categorical theory of smooth surjective morphisms which is similar to our theory of open immersions.
\end{rem}
 
\section{Local description of tensor functors}                    
\label{sec:localtens}

If a scheme $X$ is covered by two open subschemes $X_1$ and $X_2$, then we know how to describe morphisms with domain $X$. Namely, for every scheme $Y$ there is an isomorphism of sets
\[\Hom(X,Y) \myiso \Hom(X_1,Y) \times_{\Hom(X_1 \cap X_2,Y)} \Hom(X_2,Y).\]
But we can also describe morphisms with codomain $X$: If $h : Y \to X$ is a morphism, then $Y$ is covered by the two open subschemes
\[Y_1 = h^{-1}(X_1),\, Y_2 = h^{-1}(X_2)\]
so that we get morphisms $h_1 : Y_1 \to X_1$, $h_2 : Y_2 \to X_2$. These satisfy
\begin{equation*}\label{eq:a}\tag{a}
Y_1 \cap Y_2 = h_1^{-1}(X_1 \cap X_2) = h_2^{-1}(X_1 \cap X_2)
\end{equation*}
and coincide on the intersection:
\begin{equation*}\label{eq:b}\tag{b}
h_1|_{Y_1 \cap Y_2} = h_2|_{Y_1 \cap Y_2} : Y_1 \cap Y_2 \to X_1 \cap X_2.
\end{equation*}
Conversely, if $Y$ is covered by two open subschemes $Y_1,Y_2$ and two morphisms $h_1 : Y_1 \to X_1$, $h_2 : Y_2 \to X_2$ satisfy (\ref{eq:a}) and (\ref{eq:b}), then there is a unique morphism $h : Y \to X$ which induces $Y_1,Y_2,h_1,h_2$. Thus, if we have descriptions of the functors $\Hom(-,X_1)$, $\Hom(-,X_2)$ and $\Hom(-,X_1 \cap X_2)$, we also get a description of the functor $\Hom(-,X)$, i.e.\ a universal property of $X$. In the special case $X_1 \cap X_2 = \emptyset$, this is saying that the category of schemes is extensive \cite{CLW93}, but what we have just described is a notion of pushout-extensivity along open immersions.
 
We want to find a similar universal property for $\Q(X)$. Although we know that
\[\Q(X) \simeq \Q(X_1) \times_{\Q(X_1 \cap X_2)} \Q(X_2),\]
this describes $\Hom_{\c\otimes}(-,\Q(X))$, but our goal is to describe $\Hom_{\c\otimes}(\Q(X),-)$. In order to achieve this, we will prove a general result about suitable tensor categories. The bijection above then becomes an equivalence of categories. This is already the case in the situation above if $X$ is replaced by an algebraic stack.

We will mainly work with locally finitely presentable tensor categories, because the main results of \cref{sec:tensorloc} only hold for them, but with a reduction lemma (\cref{lfpred}) it will be possible to generalize some of our results to arbitrary cocomplete tensor categories.
 
\begin{defi} \label{locdat}
Let $\C,\D$ be locally finitely presentable tensor categories. We denote by
\[\Hom_{\c\otimes\fp}(\C,\D)\]
the category of cocontinuous tensor functors $\C \to \D$ which preserve finitely presentable objects (notice that the pullback functor $f^*$ of a morphism $f$ of quasi-compact quasi-separated schemes has this property). Let $I_1 \to \O_\C \leftarrow I_2$ be a finitely presentable idal cover of $\C$ (by which we mean that $I_1,I_2$ are finitely presentable). We define a category
\[\smash{\Hom_{\c\otimes\fp}^{I_1,I_2}(\C,\D)},\]
a ``localized version'' of $\Hom_{\c\otimes\fp}(\C,\D)$, as follows. An object of this category consists of the following data:
\begin{enu}
\item a finitely presentable idal cover $J_1 \to \O_\D \leftarrow J_2$ of $\D$,
\item cocontinuous tensor functors $H_1 : \C_{I_1} \to \D_{J_1}$, $H_2 : \C_{I_2} \to \D_{J_2}$ which preserve finitely presentable objects,
\item\label{eq:1} isomorphisms of idals $\pi_1 : R_{J_1} J_2 \myiso H_1(R_{I_1} I_2)$ in $\D_{J_1}$ and $\pi_2 : R_{J_2} J_1 \myiso H_2(R_{I_2} I_1)$ in~$\D_{J_2}$,
\item\label{eq:2} and an isomorphism $\delta$ between the two tensor functors $\C_{I_1 \otimes I_2} \rightrightarrows \D_{J_1 \otimes J_2}$ defined by the following diagram:
\[\begin{tikzcd}[row sep=3ex]
& (\C_{I_1})_{R_{I_1} I_2} \ar{rr}{(H_1)_{R_{I_1} I_2}} & &(\D_{I_1})_{H_1 R_{I_1} I_2} \ar{r}{\pi_1^*} \ar[double equal sign distance,shorten >=1.5ex,shorten <=1.5ex,>=stealth]{dd}[swap]{\sim} \ar[phantom]{dd}[right]{\scriptstyle\delta} & (\D_{I_1})_{R_{J_1} J_2} \ar{dr}{\sim} \\
\C_{I_1 \otimes I_2} \ar{ur}{\sim} \ar{dr}[swap]{\sim} &&&&& \D_{J_1 \otimes J_2}  \\
& (\C_{I_2})_{R_{I_2} I_1} \ar{rr}[swap]{(H_2)_{R_{I_2} I_1}} && (\D_{I_2})_{H_2(R_{I_2} I_1)} \ar{r}[swap]{\pi_2^*} & (\D_{I_2})_{R_{J_2} J_1} \ar{ur}[swap]{\sim}
\end{tikzcd} \]
\end{enu}
See \cref{formalprop} for the notation and \cref{twoidal} for the equivalences that we have used. Roughly speaking, (\ref{eq:1}) corresponds to (\ref{eq:a}), while (\ref{eq:2}) corresponds to (\ref{eq:b}). A morphism $(J_i,H_i,\pi_i,\delta) \to (J'_i,H'_i,\pi'_i,\delta')$ between two such objects consists of
\begin{enu}
\item morphisms of idals $f_i : J_i \to J'_i$ in $\D$ for $i=1,2$,
\item morphisms of tensor functors $\eta_i : H_i \to f_i^* \circ H'_i$ for $i=1,2$,
\end{enu}
\[\begin{tikzcd}[row sep=6ex]
\C_{I_i} \ar{rr}{H_i} \ar{dr}[swap]{H'_i} & \ar[double equal sign distance,shorten >=0.5ex,shorten <=0.5ex,>=stealth]{d}{\eta_i} & \D_{J_i} \\ 
& \D_{J'_i} \ar{ur}[swap]{f_i^*} & 
\end{tikzcd} \]
such that the following two coherence conditions hold. Firstly, for $i \neq j$ it is required that
\[\begin{tikzcd}[row sep=3ex]
H_i(R_{I_i} I_j) \ar{rr}{\eta_i} && f_i^*(H'_i(R_{I_i} I_j)) \\\\
R_{J_i} J_j \ar{uu}{\pi_i}[swap]{\sim} \ar{dr}[swap]{R_{J_i} f_j} & & f_i^*(R_{J'_i} J'_j) \ar{uu}{\sim}[swap]{\pi'_i} \\
 & R_{J_i} J'_j \ar{ur}[swap]{\sim} &
\end{tikzcd}\]
commutes. Secondly, we require that the composition of morphisms of tensor functors
\[\begin{tikzcd}[row sep=7.2ex, column sep=6ex]
 & (\C_{I_1})_{R_{I_1} I_2} \ar{rr}{(H_1)_{R_{I_1} I_2}} && (\D_{J_1})_{H_1(R_{I_1} I_2)} \ar{r}{\pi_1^*} \ar[double equal sign distance,shorten >=0.8ex,shorten <=0.8ex,>=stealth]{d}[left]{\sim} \ar[phantom]{d}[right]{\scriptstyle\delta}  & (\D_{J_2})_{R_{J_2} J_1} \ar{dr}{\sim} & \\
\C_{I_1 \otimes I_2} \ar{ur}{\sim} \ar{r}[swap]{\sim} & (\C_{I_2})_{R_{I_2} I_1} \ar{rr}{(H_2)_{R_{I_2} I_1}}  \ar{drr}[swap]{(H'_2)_{R_{I_2} I_1}\!\!} && (\D_{J_2})_{H_2(R_{I_2} I_1)} \ar{r}{\pi_2^*} \ar[double equal sign distance,shorten >=0.8ex,shorten <=0.8ex,>=stealth]{d}[right]{\scriptstyle(\eta_2)_{R_{I_2} I_1}} & (\D_{J_1})_{R_{J_1} J_2} \ar{r}{\sim} & \D_{J_1 \otimes J_2}\\ 
& & & (\D_{J'_2})_{H'_2(R_{I_2} I_1)} \ar{r}[swap]{{\pi'_2}^*} & (\D_{J'_2})_{R_{J'_2} J'_1} \ar{r}[swap]{\sim} & \D_{J'_1 \otimes J'_2} \ar{u}[swap]{(f_1 f_2)^*}
\end{tikzcd}\]
equals the composition of morphisms of tensor functors
\[\begin{tikzcd}[row sep=7.2ex, column sep=6ex]
  & && (\D_{J_1})_{H_1(R_{I_1} I_2)} \ar{r}{\pi_1^*} \ar[double equal sign distance,shorten >=0.8ex,shorten <=0.8ex,>=stealth]{d}[right]{\scriptstyle(\eta_1)_{R_{I_1} I_2}}  & (\D_{J_2})_{R_{J_2} J_1} \ar{r}{\sim} &  \D_{J_1 \otimes J_2}\\
\C_{I_1 \otimes I_2} \ar{dr}[swap]{\sim} \ar{r}{\sim} & (\C_{I_1})_{R_{I_1} I_2} \ar{rr}[swap]{(H'_1)_{R_{I_1} I_2}} \ar{urr}{(H_1)_{R_{I_1} I_2}\!\!}  && (\D_{J'_1})_{H'_1(R_{J'_1} I_2)} \ar{r}{{\pi'_1}^*} \ar[double equal sign distance,shorten >=0.5ex,shorten <=0.8ex,>=stealth]{d}[left]{\sim} \ar[phantom]{d}[right]{\scriptstyle\delta'} & (\D_{J'_1})_{R_{J'_1} J'_2} \ar{r}{\sim} &  \D_{J'_1 \otimes J'_2} \ar{u}[swap]{(f_1 f_2)^*} \\
 & (\C_{I_2})_{R_{I_2} I_1} \ar{rr}[swap]{(H'_2)_{R_{I_2} I_1}} && (\D_{J'_2})_{H'_2(R_{I_2} I_1)} \ar{r}[swap]{{\pi'_2}^*} & (\D_{J'_2})_{R_{J'_2} J'_1}\mydot \ar{ur}[swap]{\sim} & 
\end{tikzcd}\]
The definition of the composition of morphisms is obvious. If $\C$ and $\D$ are $\IK$-linear, we can analogously define categories $\Hom_{\c\otimes\fp\!/ \IK}(\C,\D)$ and $\smash{\Hom_{\c\otimes\fp\!/ \IK}^{I_1,I_2}(\C,\D)}$ of $\IK$-linear tensor functors.
\end{defi}

\begin{thm} \label{lokal-def}
Let $\C,\D$ be locally finitely presentable tensor categories. Let $I_1 \to \O_\C \leftarrow I_2$ be finitely presentable idal cover of $\C$. Then, there is an equivalence of categories
\[\Hom_{\c\otimes\fp}(\C,\D) \myiso \Hom_{\c\otimes\fp}^{I_1,I_2}(\C,\D).\]
A similar result holds in the $\IK$-linear case.
\end{thm}

\begin{proof}
Assume that $H : \C \to \D$ is a cocontinuous tensor functor preserving finitely presentable objects. Then the finitely presentable idals $J_i \coloneqq H(I_i)$ form a cover of $\D$, and $H$ induces cocontinuous tensor functors $H_i \coloneqq H_{I_i} : \C_{I_i} \to \D_{J_i}$ which also preserve finitely presentable objects (\cref{formalprop}). They are characterized by isomorphisms $ R_{J_i} H \myiso H_i R_{I_i}$, which evaluated at $I_j$ for $i \neq j$ yield isomorphisms $\pi_i : R_{J_i} J_j \myiso H_i(R_{I_i} I_j)$. The composition
\[\begin{tikzcd}
\C_{I_1 \otimes I_2} \ar{r}{\sim} & (\C_{I_1})_{R_{I_1} I_2} \ar{rr}{(H_1)_{R_{I_1} I_2}} && (\D_{J_1})_{H_1(R_{I_1} I_2)} \ar{r}{\pi_1^*} & \D_{R_{J_1} J_2} \ar{r}{\sim} & \D_{J_1 \otimes J_2}
\end{tikzcd}\]
is canonically isomorphic to $H_{I_1 \otimes I_2} : \C_{I_1 \otimes I_2} \to \D_{H(I_1 \otimes I_2)} \myiso \D_{J_1 \otimes J_2}$. The same holds for the composition with exchanged indices, so that we get an isomorphism of tensor functors $\delta : \pi_1^* (H_1)_{R_{I_1} I_2} \myiso \pi_2^* (H_2)_{R_{I_2} I_1}$ as required in \cref{locdat}. If $\eta : H \to H'$ is a morphism of tensor functors, then we get induced morphisms of idals $f_i \coloneqq \eta(I_i) : J_i \to J'_i$ and morphisms of tensor functors $\eta_i : H_{i} \to f_i^* H'_i$ induced by
\[\begin{tikzcd}
H_{i} R_{I_i} \ar{r}{\sim} & R_{J_i} H \ar{rr}{R_{J_i} \eta} & & R_{J_i} H' \ar{r}{\sim} & f_i^* R_{J'_i} H' \ar{r}{\sim} & f_i^* H'_i R_{I_i}.
\end{tikzcd}\]
One can check that the coherence conditions in \cref{locdat} are satisfied. This describes a~functor
\[\Hom_{\c\otimes\fp}(\C,\D) \to \Hom_{\c\otimes\fp}^{I_1,I_2}(\C,\D).\]
Conversely, let $(J_i,H_i,\pi_i,\delta)$ be an object of $\Hom_{\c\otimes\fp}^{I_1,I_2}(\C,\D)$. We let $H_{1,2} : \C_{I_1 \otimes I_2} \to \D_{J_1 \otimes J_2}$ be one of the tensor functors between which $\delta$ operates. Then $H_1,H_2,H_{1,2}$ and $\delta$ induce a~cocontinuous tensor functor
\[\C_{I_1} \times_{\C_{I_1 \otimes I_2}} \C_{I_2} \to \D_{J_1} \times_{\D_{J_1 \otimes J_2}} \D_{J_2},\]
which by \cref{global-glue} corresponds to a cocontinuous tensor functor $H : \C \to \D$. This functor preserves finitely presentable objects because of \cref{fplok}. We leave it to the reader to describe the functor
\[\Hom_{\c\otimes\fp}(\C,\D) \to \Hom_{\c\otimes\fp}^{I_1,I_2}(\C,\D)\]
on morphisms. That these functors are pseudo-inverse to each other follows essentially from \cref{global-glue}.
\end{proof}

\begin{rem}
\cref{lokal-def} in conjunction with \cref{getall} and \cite[Proposition 2.2.3]{BC14} (which is the affine case) make it theoretically possible to describe the category
\[\Hom_{\c\otimes\fp\!/ \IK}(\Q(X),\C)\]
for every quasi-compact quasi-separated $\IK$-scheme $X$ and every locally finitely presentable \mbox{$\IK$-linear} tensor category $\C$; we will see some examples below. It also follows easily that $f \mapsto f^*$ induces an equivalence
\[\Hom(Y,X) \myiso \Hom_{\c\otimes\fp\!/ \IK}(\Q(X),\Q(Y))\]
for any pair of quasi-compact quasi-separated $\IK$-schemes $X,Y$. This was already proven in \cite{BC14} without the $\fp$-condition. In particular, the category $\Hom_{\c\otimes\fp\!/ \IK}(\Q(X),\Q(Y))$ is essentially discrete (i.e.\ a setoid). Below we will prove this in a more general setting. 
\end{rem}
 
\begin{prop} \label{ess-diskret}
Let $X$ be a quasi-compact quasi-separated $\IK$-scheme and let $\C$ be a locally finitely presentable $\IK$-linear tensor category. Then the category $\Hom_{\c\otimes\fp\!/ \IK}(\Q(X),\C)$ is essentially discrete.
\end{prop}

\begin{proof}
If $X=\Spec(A)$ is affine, then
\[\Hom_{\c\otimes\fp\!/ \IK}(\Q(X),\C) \simeq \Hom_{\CAlg_{\IK}}(A,\End(\O_\C))\]
by \cite[Proposition 2.2.3]{BC14}, and the latter category is discrete. In the general case, we may assume by \cref{gen-qcqs} that $X$ is covered by two quasi-compact open subschemes $X_1$ and $X_2$ such that the claim is true for $X_1$, $X_2$ and $X_1 \cap X_2$. We choose finitely presentable idals $I_i \to \O_X$ such that $X_i = X_{I_i}$. We have to show (a) that every morphism between two objects in $\Hom_{\c\otimes\fp\!/ \IK}(\Q(X),\C)$ is an isomorphism, and (b) that any automorphism of an object equals the identity.

We first prove (a). By \cref{lokal-def} we may work in the category $\smash{\Hom_{\c\otimes\fp\!/ \IK}^{I_1,I_2}(\Q(X),\C)}$. Consider two objects $(J_i,H_i,\pi_i,\delta)$ and $(J'_i,H'_i,\pi'_i,\delta')$ in that category. A morphism is given by morphisms of idals $f_i : J_i \to J'_i$ and morphisms of tensor functors $\eta_i : H_i \to f_i^* \circ H'_i$ for $i=1,2$ such that for $i \neq j$ the diagram
\[\begin{tikzcd}[row sep=3ex]
H_i(R_{I_i} I_j) \ar{rr}{\eta_i} && f_i^*(H'_i(R_{I_i} I_j)) \\\\
R_{J_i} J_j \ar{uu}{\pi_i}[swap]{\sim} \ar{dr}[swap]{R_{J_i} f_j} & & f_i^*(R_{J'_i} J'_j) \ar{uu}{\sim}[swap]{\pi'_i} \\
 & R_{J_i} J'_j \ar{ur}[swap]{\sim} &
\end{tikzcd}\]
commutes; we will not need the other coherence diagram. Since $H_i$ and $f_i^* \circ H'_i$ are defined on $\Q(X)_{I_i} \simeq \Q(X_{I_i})$ and the claim is true for $X_{I_i}$, it follows that $\eta_i$ is an isomorphism. The diagram implies that $R_{J_i} f_j$ is an isomorphism. Consider the following commutative diagram.
\[\begin{tikzcd}[row sep=6ex, column sep=5ex]
R_{J_j} J_j \ar{rr}{R_{J_j} f_j} \ar{dr}[swap]{\sim} &&  R_{J_j} J'_j \ar{dl} \\ & \O_{\C_{J_j}} & 
\end{tikzcd}\]
It shows that $R_{J_j} J'_j \to \O_{\C_{J_j}}$ is a split epimorphism. But then it has to be an isomorphism by \cref{isiso}. Hence, the diagram shows that $R_{J_j} f_j$ is an isomorphism. Since $R_{J_i} f_j$ and $R_{J_j} f_j$ are isomorphisms, $f_i$ is an isomorphism by \cref{global-glue}, and we are done.

To prove (b), consider an automorphism $\eta : H \to H$ of $H \in \Hom_{\c\otimes\fp\!/ \IK}(\Q(X),\C)$. Let $J_i = H(I_i)$. We define an automorphism $\eta_i : H_{I_i} \to H_{I_i}$ of $H_{I_i} : \Q(X_{I_i}) \to \C_{J_i}$ by requiring that $\eta_i R_{I_i}$ equals
\[\begin{tikzcd}
H_{I_i} R_{I_i} \ar{r}{\sim} & R_{J_i} H \ar{r}{R_{J_i} \eta} & R_{J_i} H \ar{r}{\sim} & H_{I_i} R_{I_i}.
\end{tikzcd}\]
By assumption $\eta_i$ equals the identity. This means that $R_{J_i} \eta$ is the identity. By \cref{global-glue} this entails that $\eta$ is the identity.
\end{proof}

\cref{ess-diskret} can actually be generalized to arbitrary cocomplete tensor categories, using the following lemma.

\begin{lemma} \label{lfpred}
Let $\C$ be a locally finitely presentable tensor category. Let $\D$ be any cocomplete tensor category. There is an equivalence of categories
\[{\varinjlim}_{\D'} \Hom_{\c\otimes\fp}(\C,\Ind(\D')) \myiso \Hom_{\c\otimes}(\C,\D),\]
where $\D'$ runs through all essentially small subcategories of $\D$ which are closed under finite colimits and finite tensor products. Here, the Ind-category $\Ind(\D')$ is a locally finitely presentable tensor category. The corresponding statements hold in the $\IK$-linear case. 
\end{lemma}
 
\begin{proof}
Let $\C_{\fp} \subseteq \C$ denote the full subcategory of finitely presentable objects. This is an essentially small finitely cocomplete tensor category with $\C = \Ind(\C_{\fp})$. It follows that the restriction $F \mapsto F|_{\C_{\fp}}$ defines an equivalence of categories
\[\Hom_{\c\otimes}(\C,\D) \myiso \Hom_{\fc\!\otimes}(\C_{\fp},\D),\]
where $\fc$ indicates finitely cocontinuous functors. The essential image of a finitely cocontinuous tensor functor $G : \C_{\fp} \to \D$ is an essentially small subcategory of $\D$. We take its closure $\D'$ under finite tensor products and finite colimits. Its explicit recursive description shows that~$\D'$ is essentially small. Then $\D'$ is a finitely cocomplete tensor category and $G$ factors through~$\D'$. If $G \to G'$ is a morphism of tensor functors, then we find a $\D'$ which fits for both $G$ and $G'$, so that the morphism also factors through $\D'$. We deduce that the canonical functor
\[{\varinjlim}_{\D'} \Hom_{\fc\!\otimes}(\C_{\fp},\D') \to \Hom_{\fc\!\otimes}(\C_{\fp},\D)\]
is an equivalence of categories.  Finally, we observe that $G \mapsto \Ind(G)$ defines an equivalence of categories
\[\Hom_{\fc\!\otimes}(\C_{\fp},\D') \myiso \Hom_{\c\otimes\fp}(\Ind(\C_{\fp}),\Ind(\D')) \myiso \Hom_{\c\otimes\fp}(\C,\Ind(\D')). \qedhere\]
\end{proof}

\begin{cor} \label{ess-diskret2}
Let $X$ be a quasi-compact quasi-separated $\IK$-scheme. Let $\C$ be any cocomplete $\IK$-linear tensor category. Then the category $\Hom_{\c\otimes/\IK}(\Q(X),\C)$ is essentially discrete.
\end{cor}

In other words, $\Q(X)$ is an essentially codiscrete object \cite[Definition 6.1]{Sch18} in the $2$-category of cocomplete $\IK$-linear tensor categories.

\begin{proof}
This follows from \cref{ess-diskret} and \cref{lfpred}.
\end{proof}

We now give another proof of our main theorems on (fiber) products, which have the advantage that they do not rely on the existence of bicategorical coproducts resp.\ pushouts.

\begin{proof}[Alternative proof of \cref{prodthm}]
Let $X,Y$ be two quasi-compact quasi-separated $\IK$-schemes. We have to show that for every cocomplete $\IK$-linear tensor category $\C$ the canonical functor
\[\Hom_{\c\otimes/\IK}(\Q(X \times_{\IK} Y),\C) \to \Hom_{\c\otimes/\IK}(\Q(X),\C) \times \Hom_{\c\otimes/\IK}(\Q(Y),\C)\]
is an equivalence of categories. (We remark that \cref{boxpres} implies that this functor is faithful.) The affine case follows as in our first proof of \cref{prodthm} from \cite[Proposition~2.2.3]{BC14}, and in the general case we may assume $X = X_1 \cup X_2$ for two quasi-compact open subschemes such that the claim holds for the pairs $(X_1,Y)$, $(X_2,Y)$ and $(X_1 \cap X_2,Y)$. Choose finitely presentable idals $I_i \to \O_X$ with $X_i = X_{I_i}$. Let $J_i \to \O_{X \times_\IK Y}$ be the image under $p_X^* : \Q(X) \to \Q(X \times_\IK Y)$, so that $(X \times_{\IK} Y)_{J_i} = X_{I_i} \times_{\IK} Y$. We can apply \cref{lfpred} to the three locally finitely presentable $\IK$-linear tensor categories $\Q(X)$, $\Q(Y)$ and $\Q(X \times_\IK Y)$ and deduce that it is enough to prove that the canonical functor
\[\Hom_{\c\otimes\fp\!/\IK}(\Q(X \times_\IK Y),\C) \to \Hom_{\c\otimes\fp\!/\IK}(\Q(X),\C) \times \Hom_{\c\otimes\fp\!/\IK}(\Q(Y),\C)\]
is an equivalence, where $\C$ is a locally finitely presentable $\IK$-linear tensor category. By \cref{lokal-def} it is enough to prove that the canonical functor
\[\Hom_{\c\otimes\fp\!/\IK}^{J_1,J_2}(\Q(X \times_\IK Y),\C) \to \Hom_{\c\otimes\fp\!/\IK}^{I_1,I_2}(\Q(X),\C) \times \Hom_{\c\otimes\fp\!/\IK}(\Q(Y),\C)\]
is an equivalence of categories. Because of the equivalences $\Q(X \times_\IK Y)_{J_i} \simeq \Q(X_i \times_\IK Y)$ and $\Q(X)_{I_i} \simeq \Q(X_i)$ (similarly for the intersections) and \cref{locdat}, this follows from our assumption that the claim is true for $(X_1,Y)$, $(X_2,Y)$ and $(X_1 \cap X_2,Y)$.
\end{proof}

\begin{proof}[Alternative proof of \cref{fiberprodthm}]
Let $X,Y$ be quasi-compact quasi-separated schemes over some quasi-compact quasi-separated $\IK$-scheme $S$. Let $\C$ be a cocomplete $\IK$-linear tensor category. The goal is to show that the canonical functor
\[\begin{tikzcd}[row sep=3ex]
\Hom_{\c\otimes/\IK}(\Q(X \times_{S} Y),\C) \ar{d} \\ \Hom_{\c\otimes/\IK}(\Q(X),\C) \times_{\Hom_{\c\otimes/\IK}(\Q(S),\C)} \Hom_{\c\otimes/\IK}(\Q(Y),\C)
\end{tikzcd}\]
is an equivalence of categories. If $S$ is affine, this follows from \cref{prodthm}. In the general case, we may assume $S=S_1 \cup S_2$ for two quasi-compact open subschemes such that the claim is true for schemes over $S_1$, $S_2$ and $S_1 \cap S_2$. Choose finitely presentable idals $I_i \to \O_S$ with $S_i = S_{I_i}$, and let $J_i \to \O_X$, $K_i \to \O_Y$, $L_i \to \O_{X \times_S Y}$ be the images, and define $X_i \coloneqq X_{J_i}$, $Y_i \coloneqq Y_{K_i}$. We can apply \cref{lfpred} to the three locally finitely presentable $\IK$-linear tensor categories $\Q(X)$, $\Q(Y)$ and $\Q(X \times_S Y)$ and deduce that it is enough to prove that the canonical functor
\[\begin{tikzcd}[row sep=3ex]
\Hom_{\c\otimes\fp\!/\IK}(\Q(X \times_{S} Y),\C) \ar{d} \\ \Hom_{\c\otimes\fp\!/\IK}(\Q(X),\C) \times_{\Hom_{\c\otimes\fp\!/\IK}(\Q(S),\C)} \Hom_{\c\otimes\fp\!/\IK}(\Q(Y),\C)
\end{tikzcd}\]
is an equivalence, where $\C$ is a locally finitely presentable $\IK$-linear tensor category. By \cref{lokal-def} it is enough to prove that the canonical functor
\[\begin{tikzcd}[row sep=3ex]
\Hom_{\c\otimes\fp\!/\IK}^{L_1,L_2}(\Q(X \times_{S} Y),\C) \ar{d} \\ \Hom_{\c\otimes\fp\!/\IK}^{J_1,J_2}(\Q(X),\C) \times_{\Hom_{\c\otimes\fp\!/\IK}^{I_1,I_2}(\Q(S),\C)} \Hom_{\c\otimes\fp\!/\IK}^{K_1,K_2}(\Q(Y),\C)
\end{tikzcd}\]
is an equivalence of categories. Because of the equivalences $\Q(X \times_S Y)_{L_i} \simeq \Q(X_i \times_{S_i} Y_i)$, $\Q(X)_{J_i} \simeq \Q(X_i)$, $\Q(Y)_{K_i} \simeq \Q(Y_i)$, similarly for the intersections, and \cref{locdat}, this follows from our assumption that the claim is true for the $S_i$-schemes $X_i$, $Y_i$ and the $S_1 \cap S_2$-schemes $X_1 \cap X_2$ and $Y_1 \cap Y_2$.
\end{proof}
 
Finally, we describe $\Hom_{\c\otimes/\IK}(\Q(X),\C)$ for specific examples of non-affine schemes.

\begin{ex}\label{exaff}
Let $X$ be a quasi-compact quasi-separated $\IK$-scheme and let $I \to \O_X$ be a quasi-coherent idal of finite presentation. Let $\smash{\overline{X} \coloneqq X \cup_{X_I} X}$ be the gluing of two copies $X_1,X_2$ of $X$ along the quasi-compact open subscheme $X_I \subseteq X$:
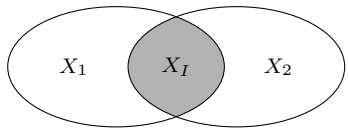
\begin{center}
\medskip
\begin{tikzpicture}[scale=0.8]
\begin{scope}
\clip (0,0) ellipse (18mm and 10mm);
\fill[black!30] (2,0) ellipse (18mm and 10mm);
\end{scope}
\draw (0,0) ellipse (18mm and 10mm);
\draw (2,0) ellipse (18mm and 10mm);
\draw node at (-0.7,0) {\scriptsize $X_1$};
\draw node at (2.7,0) {\scriptsize $X_2$};
\draw node at (1,0) {\scriptsize $X_I$};
\end{tikzpicture}
\medskip
\end{center} 
For example, when $X \coloneqq \IA^n_{\IK} = \Spec(\IK[T_1,\dotsc,T_n])$ is the $n$-dimensional affine space over $\IK$ and $I \coloneqq \langle T_1,\dotsc,T_n \rangle \hookrightarrow \O_X$, then $\smash{\overline{X}}$ is the $n$-dimensional affine space over $\IK$ with a double ``origin'' $\Spec(\IK)$. In the general case, the idal cover $I_1 \to \O_{\overline{X}} \leftarrow I_2$ corresponding to the open cover $\smash{X_1 \to \overline{X} \leftarrow X_2}$ is given by
\[I_1 |_{X_1} = \O_{X_1},~I_1 |_{X_2} = I, \quad I_2 |_{X_1} = I,~ I_2 |_{X_2} = \O_{X_2}.\]

If $\C$ is a locally finitely presentable $\IK$-linear tensor category, then by \cref{lokal-def} the category $\smash{\Hom_{\c\otimes\fp\!/ \IK}(\Q(\overline{X}),\C)}$ is equivalent to the following category: Objects are given by finitely presentable idal covers $J_1 \to \O_\C \leftarrow J_2$, tensor functors $H_i \in \Hom_{\c\otimes\fp\!/ \IK}(\Q(X),\C_{J_i})$, isomorphisms of idals $R_{J_1}(J_2) \myiso H_1(I)$, $R_{J_2}(J_1) \myiso H_2(I)$ and an isomorphism between the two associated tensor functors $\Q(X_I) \rightrightarrows \C_{J_1 \otimes J_2}$. The latter corresponds to an isomorphism between the two associated tensor functors $\Q(X) \rightrightarrows \C_{J_1 \otimes J_2}$.

Since $\C \simeq \C_{J_1} \times_{\C_{J_1 \otimes J_2}} \C_{J_2}$ by \cref{global-glue}, we can simplify this category as follows: Objects are finitely presentable idal covers $J_1 \to \O_\C \leftarrow J_2$ together with a cocontinuous $\IK$-linear tensor functor $H : \Q(X) \to \C$ preserving finitely presentable objects and an isomorphism of idals $\pi : J_1 \otimes J_2 \myiso H(I)$. We will not spell out the obvious notion of a morphism here. This is, of course, a generalization of the description of $\smash{\Hom(Y,\overline{X})}$ in terms of open covers $Y = Y_1 \cup Y_2$ and morphisms $h : Y \to X$ satisfying $h^{-1}(X_I)=Y_1 \cap Y_2$.

More generally, \cref{lfpred} implies that for every cocomplete $\IK$-linear tensor category $\C$ there is an equivalence between $\smash{\Hom_{\c\otimes/ \IK}(\Q(\overline{X}),\C)}$ and the category of cocontinuous \mbox{$\IK$-linear} tensor functors $H : \Q(X) \to \C$ with idal covers $J_1 \to \O_C \leftarrow J_2$ (not assumed to be finitely presentable) and an isomorphism of idals $\pi : J_1 \otimes J_2 \myiso H(I)$.
 
When $I$ is invertible, then $J_1$ and $J_2$ are invertible as well. In tensor categories of quasi-coherent modules every invertible object is locally free of rank $1$ and hence symtrivial \cite[Example 4.8.2]{Bra14}. Hence, the same will be true for $J_1$ and $J_2$ (since this is true in the universal example). Following James Dolan, symtrivial invertible objects have been called \emph{line objects} in \cite[Section 4.8]{Bra14}.

In the example $X=\IA^1_{\IK} = \Spec(\IK[T])$ and $I = \langle T \rangle$, when $\smash{\overline{X}}$ is the affine line with a~double origin, the idal $I \hookrightarrow \O_X$ is isomorphic to the idal $T : \O_X \to \O_X$. Moreover, the tensor functor $H : \Mod(\IK[T]) \simeq \Q(X) \to \C$ corresponds to an endomorphism $T : \O_\C \to \O_\C$. The isomorphism of idals $J_1 \otimes J_2 \myiso H(I)$ thus corresponds to an isomorphism $J_1 \otimes J_2 \myiso \O_\C$ in $\C$ such that 
\[\begin{tikzcd}[row sep=5ex, column sep=5ex]
J_1 \otimes J_2 \ar{rr}{\sim} \ar{dr} && \O_\C \ar{dl}{T} \\ & \O_\C  & 
\end{tikzcd}\]
commutes. But this makes $T : \O_\C \to \O_\C$ superfluous. Moreover, $J_1$ is invertible and hence dualizable, and the isomorphism $J_1 \otimes J_2 \myiso \O_\C$ becomes canonical if we require that it is part of a duality between $J_1$ and $J_2$.

Thus, objects of $\smash{\Hom_{\c\otimes/ \IK}(\Q(\overline{X}),\C)}$ correspond to idal covers of the form $\L \to \O_\C \leftarrow \L^*$, where $\L$ is a line object in $\C$ and $\L^*$ denotes its dual (hence inverse) object. It is routine to work out the morphisms from \cref{lokal-def}, which are unique isomorphisms by \cref{ess-diskret2}. We arrive at the following result.
\end{ex}

\begin{prop} \label{doubleorigin}
Let $\smash{\overline{\IA}_{\IK}^1}$ be the affine line over $\IK$ with a double origin. If $\C$ is a cocomplete $\IK$-linear tensor category, then the category
\[\Hom_{\c\otimes/ \IK}(\Q( \smash{\overline{\IA}_{\IK}^1}),\C)\]
is equivalent to the essentially discrete category of idal covers of the form $\L \to \O_\C \leftarrow \L^*$, where $\L \in \C$ is a line object. Two such idal covers $\L \to \O_\C \leftarrow \L^*$ and $\K \to \O_\C \leftarrow \K^*$ are isomorphic when there is an isomorphism $f : \L \to \K$ such that the diagram
\[\begin{tikzcd}[row sep={between origins,20pt}, column sep={between origins,50pt}]
\L \ar{dr} \ar{dd}[swap]{f} & & \L^* \ar{dl} \\ & \O_\C & \\
\K \ar{ur} & & \K^* \ar{ul} \ar{uu}[swap]{f^*}
\end{tikzcd}\]
commutes. \hfill $\square$
\end{prop}

For the projective line, in which two copies of $\IA^1$ are glued in a different way, we get a~similar description.

\begin{prop} \label{projline}
Let $\C$ be a cocomplete $\IK$-linear tensor category. Then the category
\[\Hom_{\c\otimes/ \IK}(\Q(\smash{\IP^1_{\IK}}),\C)\]
is equivalent to the essentially discrete category of idal covers of the form $\L \to \O_\C \leftarrow \L$, where $\L \in \C$ is a line object. Two such idal covers $\L \to \O_\C \leftarrow \L$ and $\K \to \O_\C \leftarrow \K$ are isomorphic when there is an isomorphism $f : \L \to \K$ such that
\[\begin{tikzcd}[row sep={between origins,20pt}, column sep={between origins,50pt}]
\L \ar{dr} \ar{dd}[swap]{f} & & \L \ar{dl} \ar{dd}{f} \\ & \O_\C & \\
\K \ar{ur} & & \K \ar{ul} 
\end{tikzcd}\]
commutes. Here, $\L$ corresponds to the Serre twist $\O(-1)$.
\end{prop}
 
\begin{proof}
By \cref{lfpred} we may assume that $\C$ is a locally finitely presentable $\IK$-linear tensor category and only have to study the category $\Hom_{\c\otimes\fp\!/ \IK}(\Q(\smash{\IP^1_{\IK}}),\C)$. We use the open cover $\IP^1_{\IK} = \Spec(\IK[T_1/T_0]) \cup \Spec(\IK[T_0/T_1])$ and \cref{lokal-def}. It follows that the category is equivalent to the category of finitely presentable idal covers $J_0 \rightarrow \O_\C \leftarrow J_1$ equipped with isomorphisms $\alpha : R_{J_1}(J_0) \myiso \O_{\C_{J_0}}$, $\beta : R_{J_0}(J_1) \myiso \O_{\C_{J_1}}$ (not of idals) which become inverse in $\C_{J_1 \otimes J_2}$ upon the usual identifications. This means that we have isomorphisms $\alpha : R_{J_1}(J_0) \myiso R_{J_1}(J_1)$ and $\beta^{-1} : R_{J_0}(J_0) \myiso R_{J_0}(J_1)$ which coincide in $\C_{J_1 \otimes J_2}$ and hence, by \cref{global-glue}, glue to an isomorphism $J_0 \myiso J_1$ in $\C$. That $J_i$ is a line object follows from \cref{fplok}. The rest follows from \cref{ess-diskret}.
\end{proof}

\begin{rem}
There is also a universal property for the $n$-dimensional projective space \cite[Theorem 3.1]{Bra11}. Namely, cocontinuous $\IK$-linear tensor functors $\Q(\IP^n_{\IK}) \to \C$ correspond to line objects $\L \in \C$ equipped with a morphism $s  : \L^{n+1} \to \O_\C$ such that $s$ is the coequalizer of the pair $\L^{n+1} \otimes s,\, s \otimes \L^{n+1} : \L^{n+1} \otimes \L^{n+1} \to \L^{n+1}$. See \cite[Theorem~5.10.11]{Bra14} for a generalization to projective schemes over arbitrary base schemes.
\end{rem}

For the affine plane with a double origin $\smash{\overline{\IA}_{\IK}^2}$ the idals $I_1,I_2$ in \cref{exaff} are not invertible, not even dualizable (when $\IK \neq 0$), so that the universal property becomes more complicated to state.

\begin{prop} \label{doubleorigin2}
If $\C$ is a cocomplete $\IK$-linear tensor category, then the category
\[\Hom_{\c\otimes/\IK}(\Q(\smash{\overline{\IA}_{\IK}^2}),\C)\]
is equivalent to the essentially discrete category of idal covers $J_1 \rightarrow \O_\C \leftarrow J_2$ and morphisms $p : \O_\C^2 \to J_1 \otimes J_2$ such that $p$ is the cokernel of $(y;-x) : \O_\C \to \O_\C^2$, where $(x,y) : \O_\C^2 \to \O_\C$ is defined as the composition $\O_\C^2 \xrightarrow{\,p\,} J_1 \otimes J_2 \to \O_\C$.
\end{prop}

\begin{proof}
This follows from \cref{exaff} and the exact sequence of $\IK[T_1,T_2]$-modules
\[\begin{tikzcd}[column sep=8ex]
\IK[T_1,T_2] \ar{r}{(T_2;-T_1)} &  \IK[T_1,T_2]^2 \ar{r}{(T_1,T_2)} & \langle T_1,T_2 \rangle \ar{r} & 0.
\end{tikzcd} \qedhere\] 
\end{proof}

We have seen in our proofs of \cref{prodthm}, \cref{fiberprodthm} and \cref{ess-diskret2} how to use idal covers to reduce theorems about the cocomplete tensor category of quasi-coherent modules on a scheme to the affine case. Let us sketch another application of this type.

\begin{rem} \label{tangentcat}
In \cite[Section 5.2]{Bra14} tangent tensor categories have been defined as follows. Let $\Cat_{\c\otimes/\IK}$ denote the $2$-category of cocomplete $\IK$-linear tensor categories. Let $\C \in \Cat_{\c\otimes/\IK}$ and consider the coslice $2$-category $\C \backslash {\Cat_{\c\otimes/\IK}}$ with its $2$-endofunctor $\A \mapsto \A[\varepsilon]/\varepsilon^2$. Now if $\C \backslash \D  \in \C \backslash {\Cat_{\c\otimes/\IK}}$ is an object (i.e.\ a morphism $\C \to \D$), then its \emph{tangent tensor category} $T(\C \backslash \D) \in \C \backslash {\Cat_{\c\otimes/\IK}}$ is defined by the universal property
\[\Hom_{\C \backslash {\Cat_{\c\otimes/\IK}}}\bigl(T(\C \backslash \D),\A\bigr) \simeq \Hom_{\C \backslash {\Cat_{\c\otimes/\IK}}}\bigl(\C \backslash \D, \A[\varepsilon]/\varepsilon^2\bigr).\]
The results of \cite{Bra20} can be used to show that $T(\C \backslash \D)$ exists when $\C,\D$ are locally finitely presentable tensor categories and $\C \to \D$ preserves finitely presentable objects.

If $X \to S$ is a morphism of schemes which is affine or projective, then we have
\[T\bigl(\Q(S) \backslash {\Q(X)}\bigr) \simeq \Q\bigl(T(X/S)\bigr)\]
by \cite[Theorem 5.12.17]{Bra14}, where $\smash{T(X/S) \coloneqq \Spec \Sym \Omega^1_{X/S} \in {\Sch}/S}$ is the usual tangent bundle. The proof of the projective case in loc.cit.\ is quite long-winded and uses a tensor categorical Euler sequence. \mbox{However,} using \cref{lokal-def} and \cref{global-glue} as we did before, it is straight forward to generalize this equivalence from the affine case to every morphism $X \to S$ of quasi-compact quasi-separated schemes. 
\end{rem}


\end{document}